\documentclass[12pt]{amsart}

\usepackage{times}      
\usepackage{amssymb}    
\usepackage{amsmath}    
\usepackage{amsthm}
\usepackage{graphicx}
\usepackage{hyperref}
\usepackage{upgreek}
\usepackage{enumitem}
\usepackage{caption}
\usepackage{multirow}
\usepackage{array}

\usepackage[normalem]{ulem}

\makeatletter
\g@addto@macro\bfseries{\boldmath}
\makeatother

\makeatletter
\def\th@plain{%
	\thm@notefont{}
	\itshape 
}
\def\th@definition{%
	\thm@notefont{}
	\normalfont 
}
\makeatother

\usepackage[all]{xy}
\xyoption{all}
\xyoption{arc}
\xyoption{poly}

\usepackage{tikz}

\usepackage{tkz-graph}
\SetUpVertex[Lpos=270,Math,LabelOut]	
\GraphInit[vstyle=Classic]				
\tikzset{VertexStyle/.append style = { minimum size = 4pt }}	
\tikzset{EdgeStyle/.style = {-,thick}}

\newtheorem{theorem}{Theorem}[section]
\newtheorem{mainthm}[theorem]{Main Theorem}

\newtheorem{conjecture}[theorem]{Conjecture} 
\newtheorem{lemma}[theorem]{Lemma} 
\newtheorem{proposition}[theorem]{Proposition} 
\newtheorem{remark}[theorem]{Remark}

\theoremstyle{definition}
\newtheorem{definition}[theorem]{Definition} 
\newtheorem{example}[theorem]{Example} 
\newtheorem*{acknowledge}{Acknowledgments}
\newtheorem*{organization}{Organization}

\def\bdm{\begin{displaymath}}
\def\edm{\end{displaymath}}

\def\rrr{\mathbb{R}}
\def\ccc{\mathbb{C}}
\def\zzz{\mathbb{Z}}
\def\qqq{\mathbb{Q}}

\def\sph{\mathrm{S}}
\def\disk{\mathrm{D}}
\def\rp{\rrr \mathrm{P}}
\def\cp{\ccc \mathrm{P}}
\def\wcp{\cp ^2 _{a,b,c}}
\def\O{\mathrm{O}}
\def\SO{\mathrm{SO}}
\def\SL{\mathrm{SL}}
\def\T{\mathrm{T}}

\def\d{\partial}
\def\id{\mathrm{id}}

\def\normal{\vartriangleleft}
\def\st{\mathrel{}\middle|\mathrel{}}

\def\subgp{\subset}

\def\q{Q}
\def\qprime{\q\ensuremath{'}}

\DeclareMathOperator{\xt}{xt}
\DeclareMathOperator{\dist}{dist}
\DeclareMathOperator{\diam}{diam}
\DeclareMathOperator{\curv}{curv}
\DeclareMathOperator{\cl}{cl}
\DeclareMathOperator{\Fix}{Fix}
\DeclareMathOperator{\Susp}{Susp}

\newcommand{\I}[2]{\# \left( #1 \cap #2 \right)} 

\let\originalleft\left
\let\originalright\right
\renewcommand{\left}{\mathopen{}\mathclose\bgroup\originalleft}
\renewcommand{\right}{\aftergroup\egroup\originalright}

\begin{document}

\title[Positively curved orbifolds with circle symmetry]{Positively curved Riemannian orbifolds and Alexandrov spaces with circle symmetry in dimension $4$}

\author[Harvey]{John Harvey}
\address{Department of Mathematics, Swansea University, Fabian Way, Swansea, U.K.}
\email{john.harvey.math@gmail.com}

\author[Searle]{Catherine Searle}
\address{Department of Mathematics, Statistics, and Physics, Wichita State \, University, Wichita, KS 67260, USA.}
\email{searle@math.wichita.edu}

\subjclass[2010]{Primary: 53C23; Secondary: 51K10, 53C20, 57R18}

\date{Oct 16, 2021}

\begin{abstract} We classify positively curved Alexandrov spaces of dimension $4$ with an isometric circle action up to equivariant homeomorphism, subject to a certain additional condition on the infinitesimal geometry near fixed points which we conjecture is always satisfied.
	
As a corollary, we also classify positively curved Riemannian orbifolds of dimension $4$ with an isometric circle action.
\end{abstract}
\maketitle


\section{Introduction}

 In their seminal work, Hsiang and Kleiner~\cite{HK} showed that for a closed, orientable, positively curved Riemannian manifold $M$ of dimension $4$, admitting an isometric action of the circle $T^1$,  $M$ is homeomorphic to $\sph^4$ or $\cp^2$. This classification was obtained via a study of the structure of the orbit space of such a circle action and an analysis of the fixed-point sets of that action, combined with the homeomorphism classification of $4$-manifolds due to  Freedman~\cite{Fr}. 

The orbit space $M/T^1$ is a positively curved Alexandrov space, and our understanding of these spaces is much improved since the publication of~\cite{HK}. By making use of Alexandrov geometry, as well as the resolution of the Poincar\'{e} Conjecture, this classification has been strengthened to show that the manifolds are actually equivariantly diffeomorphic to $\sph^4$ or $\cp^2$ with a linear action by work of  Grove and Searle~\cite{GS1} and Grove and Wilking~\cite{GW2}.

The nature of these results can be probed further by relaxing certain of the hypotheses. If we assume only non-negative sectional curvature, it was shown independently by Kleiner~\cite{K} and Searle and Yang~\cite{SY} that we add in, up to homeomorphism, only $\sph^2 \times \sph^2$ and $\cp^2 \# \pm \cp^2$. Galaz-Garc\'ia~\cite{GG}, Galaz-Garc\'ia and Kerin~\cite{GGK}, and~\cite{GW2} showed that this too can be improved to equivariant diffeomorphism.

Alternatively, one can relax the assumption that the spaces are Riemannian manifolds.
In~\cite{Y}, Yeroshkin considers the case of positively curved Riemannian orbifolds.
With the additional assumption that the orbifold fundamental group is trivial, he proves that the underlying topological space of the orbifold is homotopy equivalent to $\sph^4$ or has the cohomology of $\cp^2$.
In the case where there is a $2$-dimensional fixed-point set, he improves the classification to homeomorphism, showing that it is $\sph^4$ or a weighted complex projective space, $\cp^2_{a, b, c}$. He conjectures that this also holds true when there are only isolated fixed points (see Conjecture 5.3~\cite{Y}).

We resolve this conjecture in the affirmative.

\begin{theorem}[Orbifold Classification]\label{t:orbifold}
	Let $T^1$ act isometrically and effectively on $X^4$, where $X^4$ is a $4$-dimensional, closed, positively curved, orientable Riemannian orbifold, with $\pi_1^{\mathrm{orb}}(X)=\lbrace 0 \rbrace$. Then, up to equivariant homeomorphism, the underlying topological space $\left| X \right|$ is one of the following spaces:
	\begin{enumerate}
		\item The $4$-sphere, with a linear action; or
		\item A weighted complex projective space, with an action induced by a linear $T^2$ action on $\sph^5$.
	\end{enumerate}
\end{theorem}

This result is obtained as an immediate corollary of a more general result, Theorem \ref{t:amsr}, which relaxes further the hypothesis on the space to permit Alexandrov spaces, a class of spaces that include  Riemannian orbifolds.
Symmetries of low-dimensional Alexandrov spaces have been studied elsewhere (see work of  N\'u\~{n}ez-Zimbr\'on \cite{NZ} in dimension 3 and Corro,  N\'u\~{n}ez-Zimbr\'on and Zarei \cite{CNZZ} and Galaz-Garc\'{i}a \cite{GG2} in dimension 4).

A crucial element of the Riemannian results described above is the determination of an upper bound on the number of isolated fixed points in the space.
This is established by an elegant argument which relies on an understanding of the geometry of the space of directions at these fixed points in the orbit space.
In particular, the space of directions contains no triangles with perimeter exceeding $\pi$, or, in the language of extents, the $3$-extent is  bounded above by $\tfrac{\pi}{3}$. In this paper, we refer to such spaces of directions as \emph{small}.

The proof of the upper bound on the $3$-extent is straightforward in the Riemannian cases, both for manifolds and orbifolds, since the geometry of the space of directions is rigid.
The greater flexibility in general Alexandrov spaces creates a surprising difficulty in proving the upper bound in full generality. 
We define here a condition which permits the extension of this bound.

\begin{definition}[Condition \qprime{}] An isometric action of the circle $T^1$ on an Alexandrov space $\Sigma^3$ of dimension $3$ with $\curv \geq 1$ is said to satisfy \emph{Condition \qprime{}} if it is fixed-point-free and the following hold:
	
	\begin{enumerate}
		\item $\Sigma^3/T^1$ is a small space;
		\item The double branched cover of $\Sigma^3 / T^1$ over any two points corresponding to finite isotropy is small; and
		\item If there are three components of finite isotropy, $\diam \left( \Sigma^3 / T^1 \right) \leq \tfrac{\pi}{4}$.
	\end{enumerate} 	
\end{definition}

The first part of the condition is sufficient to bound the number of isolated fixed points. The other two parts are used to obtain that  the singular set is unknotted in the orbit space. 

While Condition \qprime{} may seem technical, we show in Lemma \ref{l:qprime} that it is satisfied for any fixed-point-free isometric circle action on
a $3$-dimensional spherical orbifold of constant curvature 1.

\begin{definition}[Condition \q{}] An isometric action of the circle $T^1$ on a $4$-dimensional Alexandrov space is said to satisfy \emph{Condition \q{}} if at every isolated fixed point the isotropy action satisfies Condition \qprime{}.
\end{definition}

For example, any isometric circle action on a Riemannian $4$-orbifold satisfies Condition \q{} by Lemma \ref{l:qprime}.

\begin{mainthm} \label{t:amsr} 
	 Let $T^1$ act isometrically and effectively on $X^4$ so as to satisfy Condition \q{}, where $X^4$ is a $4$-dimensional, closed, positively curved, orientable Alexandrov space. Then, up to equivariant homeomorphism, $X$ is one of the following spaces:
	\begin{enumerate}
		\item The suspension of a spherical $3$-manifold, with a linear action; or
		\item A finite quotient of a weighted complex projective space with a linear action.
	\end{enumerate}
\end{mainthm}

Let us immediately point out how Theorem \ref{t:orbifold} follows from this result. First, Theorem \ref{t:amsr} is applicable, since a  closed, positively curved, orientable Riemannian orbifold is an example of a closed, positively curved, orientable Alexandrov space satisfying Condition \q{}. The restriction on the orbifold fundamental group simply means that we exclude any finite quotients from the classification, so that the spherical $3$-manifold can be taken to be $\sph^3$.

Note that this list of possible spaces in Theorem \ref{t:amsr} is very restrictive, with every space being the quotient of a sphere. The additional spaces obtained by relaxing the Riemannian hypothesis arise only because the class of Alexandrov spaces is closed under taking the quotient by an isometric group action, even when that action is not free.

As in the manifold case, the bound on the symmetry rank of a positively curved $n$-dimensional Alexandrov space is $\lfloor \tfrac{n+1}{2}\rfloor$. It is interesting to compare Theorem \ref{t:amsr} to the following result of the authors~\cite{HS}, which shows that all spaces that achieve this bound are obtained as quotients of spheres.

\begin{theorem}[Maximal Symmetry Rank Theorem]\cite{HS}
	\label{t:alexmsr} Let $X$ be an $n$-dimensional, compact,
 Alexandrov space with $\curv \geq 1$ admitting an isometric $T^k$ action with $k=\lfloor  \tfrac{n+1}{2} \rfloor$. Then either
\begin{enumerate}
\item $X$ is a spherical orbifold, homeomorphic to $\sph^n / G$, where $G$ is a finite subgroup of the centralizer of the maximal torus in $\O(n+1)$; or
\item Only in the case that $n$ is even, $X \simeq \sph^{n+1}/G$, where $G$ is a rank one subgroup of the maximal torus in $\O(n+2)$
\end{enumerate}
and in both cases the action on $X$ is equivalent to that induced by the maximal torus.
\end{theorem}

We note that while for positively curved Riemannian manifolds there is no difference in dimension $4$ between the maximal and almost maximal symmetry rank cases, there is a difference for Alexandrov spaces. Namely, while the spaces 
 obtained in Theorems \ref{t:amsr} and \ref{t:alexmsr} are all linear quotients of $\sph^4$ and $\sph^5$, in Theorem \ref{t:amsr} there are a greater variety of quotients possible, as the dihedral and binary polyhedral subgroups of $\SO(5)$ and $\SO(6)$ also occur. 

\begin{organization}
The paper is organized as follows. In Section \ref{s:prelim}, we present notation and conventions, as well as background material about Alexandrov spaces, group actions on Alexandrov spaces, and Seifert manifolds. In Section \ref{s:singularities},  we describe the restrictions imposed by positive curvature on isolated singular 
points and singular knots. In Section \ref{s:top3points}, we prove a topological classification when there are three isolated points of circle isotropy. In Section \ref{s:amsr3} we classify isometric circle actions on positively curved $3$-spaces and in Section \ref{s:amsr4} we address the $4$-dimensional case, proving Theorem \ref{t:amsr}. In Section \ref{s:alex}, we discuss a conjecture about the extents of quotients of Alexandrov spaces by isometric group actions which would show that Condition \q{} is always satisfied.
\end{organization}

\begin{acknowledge}
J.~Harvey's work on this project was supported in part by a Daphne Jackson Fellowship sponsored by the UK Engineering and Physical Sciences Research Council and Swansea University, as well as by Karsten Grove's grant from the National Science Foundation and Burkhard Wilking's Leibniz-Preis from the Deutsche Forschungsgemeinschaft.
	
C.~Searle is grateful to the Mathematics Department of the University of Notre Dame for its hospitality during a visit where a part of this research was carried out. 
She was supported in part by Simons Foundation Grant \#355508 (Catherine Searle), her NSF grant numbers DMS-160178 and DMS-1906404 and CONACyT Project numbers SEP-CO1-46274 and SEP-82471.   
This material is based in part upon work supported by the National Science Foundation under Grant number DMS-1440140 while C.~Searle was in residence at the Mathematical Sciences Research Institute in Berkeley, California, during the Spring 2016 semester.\enlargethispage{\baselineskip}

Both authors are grateful to B.~Wilking, A.~Petrunin, and F.~Wilhelm for conversations around the possibility of removing Condition \q{} from Theorem \ref{t:amsr} and
to Orsola Capovilla-Searle for her help with the figures.

Finally, both authors express their gratitude to the anonymous reviewer who suggested a sharper focus for this article and whose comments led us to simpler arguments for Propositions \ref{p:3space} and \ref{p:loopandspur}.
\end{acknowledge}


\section{Preliminaries}\label{s:prelim}

In this section we  fix notation, recall basic definitions and theorems about Alexandrov spaces, as well as discuss  how unnormalized invariants are assigned to Seifert manifolds.


\subsection{Alexandrov geometry} \label{s:alexandrov}

A finite-dimensional \emph{Alexandrov space} is a locally complete, locally compact, connected (except in dimension $0$, where a two-point space is admitted) length space, with a lower curvature bound in the triangle comparison sense. We assume throughout that the space is compact, and usually without boundary, in which case it is said to be \emph{closed}. There are a number of introductions to Alexandrov spaces to which the reader may refer for basic information (see, for example, Burago, Burago, and Ivanov~\cite{BBI}, Burago, Gromov, and Perelman~\cite{BGP}, Plaut~\cite{Pl}, and Shiohama~\cite{Sh}). Riemannian orbifolds are simple examples of Alexandrov spaces: in fact,  $4$-dimensional Alexandrov spaces  are all homeomorphic to orbifolds. 

The \emph{space of directions} of an Alexandrov space $X^n$ of dimension $n$ at a point $p$ is,
by definition, the completion of the 
space of geodesic directions at $p$ and is denoted by $\Sigma_p X$ or, when there is no confusion, $\Sigma_p$.  
It is a compact Alexandrov $(n-1)$-dimensional space with $\curv \geq 1$. A small metric ball around $p$, $B_r(p)$, is homeomorphic to an open cone on $\Sigma_p$ by work of Perelman~\cite{P}. The local model for an Alexandrov space is therefore given by a cone on any space of positive curvature. In the case of a Riemannian orbifold, $\Sigma_p$ is isometric to a quotient of the unit sphere by a finite group.

The class of Alexandrov spaces is closed under taking quotients by isometric group actions, even when those actions are not free. Furthermore, the subclass of spaces with $\curv \geq 1$ is closed under taking spherical suspensions and spherical joins.

An Alexandrov space has an open dense subset which is a topological manifold, and in the event that the space has no boundary then the complement of the manifold part has codimension at least three.
In low dimensions, the structure is then relatively simple.
In dimension three, the only topological singularities are isolated points with space of directions homeomorphic to $\rp^2$.
In fact, Galaz-Garc\'{i}a and Guijarro~\cite{GGG2} have classified the positively curved $3$-spaces as follows.

\begin{proposition}\label{p:GGG}
	 Any positively curved $3$-space is homeomorphic to a quotient of $\sph^3$ by some finite subgroup $\Gamma$ of $\O(4)$. In particular, it is homeomorphic to the suspension of $\rp^2$ or to a spherical manifold.
\end{proposition}

By Lemma 3.3 in~\cite{HS},  an Alexandrov space without boundary, $X$,  is orientable if and only if its manifold part is orientable. If, for some $p \in X$, $\Sigma_p$ is not orientable, then $X$ does not even admit a local orientation near $p$. In general, if the space is not orientable, it may be obtained as the quotient of an orientable Alexandrov space by an isometric involution by Theorem 3.4 in \cite{HS} (cf. \cite{GGG2} for the $3$-dimensional case).

In an oriented Alexandrov space $X$,  the {\em  intersection number} of two subsets $A$ and $B$, $\I{A}{B}$, is defined as usual on the manifold part of $X$. In this paper, $A$ and $B$  are either two curves in a surface, or  a curve and a surface in a $3$-dimensional space.

We recall Petrunin's  analogue of Synge's Theorem ~\cite{Pet1} for Alexandrov spaces, giving here the version  from~\cite{HS}.

\begin{theorem}[Generalized Synge's Theorem]\label{GST} Let $X$ be an Alexandrov space of dimension $n$ with $\curv \geq 1$. 
\begin{enumerate}
\item Let $X$ be even-dimensional. If $X$ is either orientable or locally non-orientable
 then $X$ is simply connected, and otherwise it has fundamental group $\zzz_2$.
\item If $X$ is odd-dimensional and locally orientable then $X$ is orientable.
\end{enumerate}
\end{theorem}

\subsection{Equivariant Alexandrov geometry}\label{s:equivalex}

We  now concentrate our attention on isometric group actions on Alexandrov spaces. 

Given an isometric (left) action $G\times X\rightarrow X$ of a Lie group $G$, and a point $p\in X$, we let $G(p)=\{\,gp :g\in G \,\}$ be the \emph{orbit}  of $p$ under the action of $G$. The \emph{isotropy group} of $p$ is the subgroup $G_p=\{\, g\in G: gp=p\,\}$. Recall that $G(p)\cong G/G_p$. We  denote the orbit space of this action by $\bar{X}=X/G$. Similarly, the image of a point $p\in X$ under the orbit projection map $\pi:X\rightarrow \bar{X}$
 is denoted by $\bar{p}\in \bar{X}$. We  assume throughout that $G$ is compact and its action is \emph{effective}, that is, that $\bigcap_{p\in X}G_p$ is the trivial subgroup $\{e\}$ of $G$. By Theorem 2.2 in Galaz-Garc\'{i}a and Guijarro \cite{GGG1}, the set of principal orbits  then forms an open dense subset of $X$.

As in the case of Riemannian manifolds, the space of directions at any point $p$ decomposes as the spherical join of the orbital directions (the unit sphere in the Lie algebra) and the normal directions $\nu_p$, which in general might be any Alexandrov space with $\curv \geq 1$, see Galaz-Garc\'ia and Searle~\cite{GGS}. If $G$ acts effectively on $X$ then the induced isometric action of  $G_p$ on $\Sigma_p$ must be effective. Where $G_p \normal G$, it follows that the $G_p$ action on $\nu_p$ must also be effective.

We recall the following results from~\cite{HS}.

\begin{theorem}[Slice Theorem] \label{l:slice}\cite{HS} Let $G$, a compact Lie group, act isometrically on an Alexandrov space $X$. Then for all $p \in X$, there is some $r_0 > 0$ such that for all $r < r_0$ there is an equivariant homeomorphism $\Phi \colon G \times_{G_p} K\nu_p \rightarrow B_r(G(p))$ where $\nu_p$ is the space of normal directions to the orbit $G(p)$.\end{theorem}

\begin{lemma}\label{l:fixedpoint} \cite{HS} Let $T^k$ act by isometries on $X^{2n}$, a compact even-dimensional Alexandrov space of positive curvature. Then $T^k$ has a fixed point.\end{lemma}

\begin{proposition}\label{p:evencodim}\cite{HS} Let $T^1$ act isometrically and effectively on $X^n$, a compact Alexandrov space. Then the components of the fixed-point set are of even codimension in $X^n$. \end{proposition}

An action of $G$ on $X$ such that $X^G \neq \emptyset$ and $\dim(X^G) = \dim (X/G) -1$ is called \emph{fixed-point homogeneous}. The following description of positively curved fixed-point-homogeneous Alexandrov spaces from \cite{HS} is very useful in our classification. Since $\nu$ might be any positively curved space, this result of the authors shows that, in marked contrast to the Riemannian setting, fixed-point homogeneity is not a very restrictive concept in Alexandrov geometry.

\begin{theorem}[Fixed-point-homogeneous actions]\label{t:fph}\cite{HS} 
Let $G$, a compact Lie group, act isometrically and fixed-point-homogeneously on $X$, an Alexandrov space with $\curv \geq 1$.
If $F$ is the component of $X^G$ with maximal dimension then the following hold:
\begin{enumerate}
\item There is a unique orbit $G(p)\cong G/G_p$ at maximal distance from $F$ (the ``soul" orbit);
\item  The space $X$ is homeomorphic to  
$(\nu *G)/G_p$, where $\nu$ is the space of normal directions to the orbit at $p$ and $G_p$ acts on the left on $\nu*G$, the action on $\nu$ being the isotropy action at $p$ and the action on $G$ being the inverse action on the right; and
\item The above homeomorphism is in fact $G$-equivariant, where the action of $G$ on $(\nu *G)/G_p$ is induced by the left action on $\nu * G$ given by the join of the trivial action and the left action.
\end{enumerate}
\end{theorem}

Consider the subset $\bar{X}_{(H)} \subset \bar{X}$ which is the image of all orbits with isotropy subgroup conjugate to $H \subgp G$. Perelman and Petrunin~\cite{PePet} have shown that its closure $\cl \left( \bar{X}_{(H)} \right)$ is an \emph{extremal subset} of $\bar{X}$, that is, a closed subset which is preserved under the gradient flow of $\dist(p,\cdot)$ for all $p \in \bar{X}$. Extremal sets stratify an Alexandrov space into topological manifolds. Extremal sets of codimension one make up the boundary. A primitive $1$-dimensional extremal set, that is, one which cannot be expressed as a union of proper subsets which are also extremal, is a curve, which is either closed or terminates in $0$-dimensional extremal sets. 

It is clear from the Relative Stability Theorem of Kapovitch~\cite{K} that, if the Alexandrov space $X$ is a topological manifold and $E$ is an extremal subset, the top stratum of $E$, that is, the complement in $E$ of any strictly smaller extremal subsets, is a locally flat submanifold (cf.~\cite[Lemma 5.1]{GW2}). Recall that a \emph{locally flat submanifold} is the topological analogue of an embedded submanifold in the differential category. More precisely, if $M^n$ is a manifold then $N^k \subset M^n$ is locally flat if, for each $p \in N$, there is a neighborhood $U$ of $p$ so that $(U, U \cap N) \cong (\rrr^n, \rrr^k)$, where $\rrr^k$ is included in $\rrr^n$ in the standard way.

\subsection{Assigning unnormalized invariants to Seifert manifolds}\label{s:seifert}

In Section \ref{s:invariants} we  assign invariants to $T^1$-actions on orientable Alexandrov $4$-spaces. The $4$-spaces under consideration there have isolated fixed points. Let $p$ be any such isolated fixed point. Then Proposition \ref{p:GGG} implies $\Sigma_p$  is homeomorphic to $\sph^3 / \Gamma$. The isotropy action defines a Seifert fibration on $\Sigma_p$.

For this reason, Seifert fibrations are key to defining the invariants we  need to classify these $4$-spaces, and so we establish here notational conventions for the invariants of Seifert manifolds. 
Note that we do not include the Seifert invariants related to orientability or the presence of fixed points, since these are not  relevant here.

We  follow the approach of Jankins and Neumann~\cite{JN}, since it is better suited to the present work than the normalized Seifert invariants which were used by Fintushel~\cite{F1} to classify $4$-manifolds with circle actions. We collect all of the notation into Table \ref{table} at the end of this subsection for easy reference.

Let $T^1$ act on $Z^3$, an oriented $3$-manifold, without fixed points.  Let $\pi \colon Z \to \bar{Z}=Z/T^1$ be the projection map. Let $g$ be the genus of $\bar{Z}$, which for the purposes of this paper is always zero. Orient $\bar{Z}$ so that its orientation followed by the orientation of the orbits gives the orientation of $Z^3$.

Let $\left\{ \bar{x}_1, \ldots, \bar{x}_n \right\} \subset \bar{Z}$ be the image of a non-empty collection of orbits, which includes all of the exceptional orbits. In fact, it is sufficient to choose only the exceptional orbits unless this produces an empty collection, in which case the image of one principal orbit suffices. Choose disjoint $2$-disk neighborhoods, $\bar{V}_i$, of each $\bar{x}_i$. Orient $\partial \bar{V}_i$ by its inward normal. The closure of the complement of the disks is denoted by $\bar{R} = \cl \left( \bar{Z} \setminus \left( \bigcup_{i=1}^n \bar{V}_i \right) \right)$ and is the surface of genus $g$ with $n$ boundary components. See Figure \ref{f:zbar}.

\begin{figure}
	\begin{tikzpicture}
	\node[anchor=south west,inner sep=0] (image) at (0,0) {\includegraphics[width=0.8\textwidth]{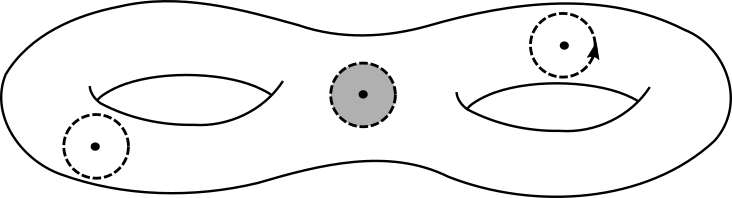}};
	\begin{scope}[
	x={(image.south east)},
	y={(image.north west)}
	]
	\node  at (0.23,0.13) {$\bar{x}_1$};
	\draw [->, thick] (0.2,0.16) to [bend right] (0.14,0.26);
	\node  at (0.4,0.3) {$\bar{V}_2$};
	\draw [->, thick] (0.4,0.4) to [bend left=60] (0.48,0.46);
	\node  at (0.86,0.73) {$\partial \bar{V}_3$};
	\node  at (0.3,0.8) {$\bar{R}$};
	\end{scope}
	\end{tikzpicture}
	\caption{The orbit space $\bar{Z}$ in the case $g=2$, $n=3$.}
	\label{f:zbar}
\end{figure}
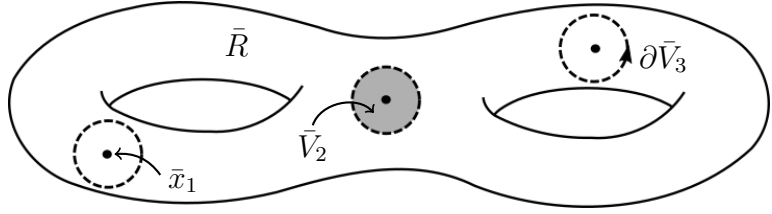

The preimage of the disk, $V_i = \pi^{-1} (\bar{V}_i) \cong \disk^2 \times \sph^1$, is a solid torus. Orient the boundary torus $\partial V_i \cong \T^2$ by its inward normal. 
Let $x_i \in \pi^{-1}(\bar{x}_i)$, the core of $V_i$. Denote  the slice at $x_i$ by $S_i \cong \disk^2$, and note that $\pi(S_i) = \bar{V}_i \cong \disk^2$ is the quotient of $S_i$ by the finite cyclic group $T^1_{x_i}$. Orient $S_i$ so that its intersection number with $T^1(x_i)$ is $\I{S_i}{T^1(x_i)}=+1$. 

Let $m_i$ be the boundary curve of $S_i$, oriented by its inward normal. This curve is a meridian of  $\partial V_i$, that is, a curve representing a generator of $H_1 (\partial V_i ; \zzz)$ which bounds a disk in $V_i$. Let $l_i$ be a longitudinal curve on $\partial V_i$, chosen so that $\I{l_i}{m_i}=+1$. Note that this does not uniquely specify $l_i$, which can be varied by adding any multiple of $m_i$. Further,  we  freely identify oriented curves with the corresponding elements of $H_1 (\partial V_i ; \zzz)$. 

The pair $\left( l_i, m_i \right)$ now forms an oriented basis for $H_1 (\partial V_i ; \zzz)$. The homology class of any other closed curve $t$ in $\partial V_i$ may be written  as $t = a l_i + b m_i$ for some $a, b \in \zzz$. The coefficients then give the intersection numbers $\I{t}{m_i} = a$ and $\I{t}{l_i} = -b$.

Let $h_i$ be an oriented principal orbit on $\partial V_i$. Writing $h_i = \alpha_i l_i + \gamma_i m_i$, we have that $\I{h_i}{m_i}=\alpha_i$ is the order of the isotropy group at $x_i$ (see Figure \ref{f:lmh}) and this gives one Seifert invariant for the orbit over $\bar{x}_i$. 

\begin{figure}
	\begin{tikzpicture}
	\node[anchor=south west,inner sep=0] (image) at (0,0) {\includegraphics[width=0.7\textwidth]{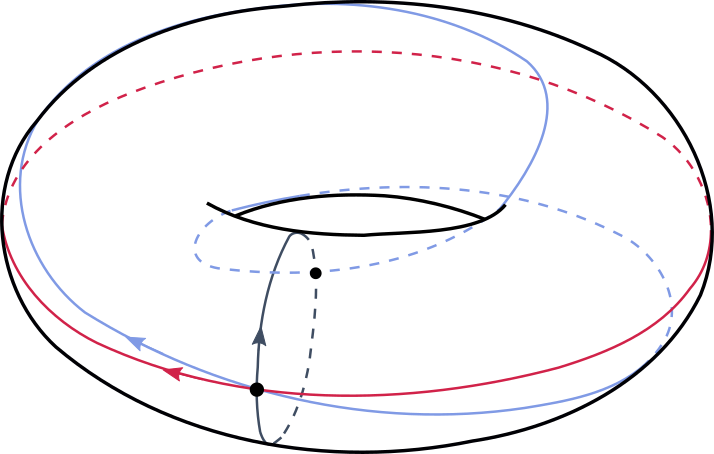}};
	\begin{scope}[
	x={(image.south east)},
	y={(image.north west)}
	]
	\node  at (0.24,0.13) {$l_i$};
	\node  at (0.475,0.26) {$m_i$};
	\node  at (0.08,0.5) {$h_i$};
	\end{scope}
	\end{tikzpicture}
	\caption{A meridian, $m_i$, 
	 longitude, $l_i$, 
	 and principal orbit, $h_i$, 
	  on $\partial V_i$. 
	The intersection number of $h_i$ with $m_i$ is $\alpha_i = 2$, 
	the order of the isotropy at the exceptional orbit. 
	}
	\label{f:lmh}
\end{figure}
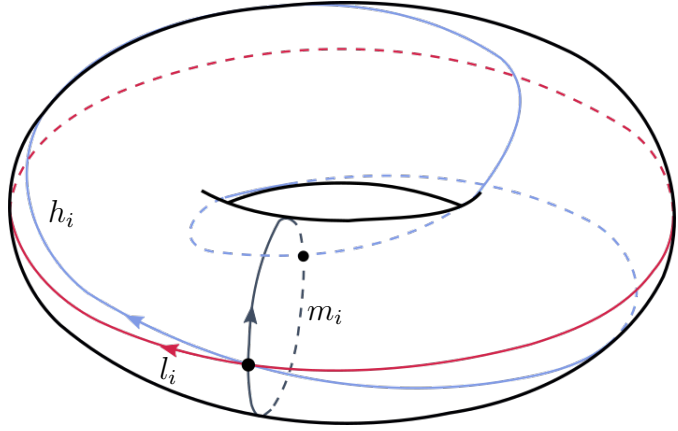

We remove the neighborhoods $\bar{V}_1, \ldots, \bar{V}_n$ from $\bar{Z}$ to obtain the surface with boundary $\bar{R} = \cl \left( \bar{Z} \setminus \left( \bigcup_{i=1}^n \bar{V}_i \right) \right)$. Recall that principal $T^1$ bundles over a base $B$ are classified by $H^2(B)$, so all principal $T^1$ bundles over surfaces with boundary are trivial. Thus, we can
choose a section, $\sigma \colon \bar{R} \to R$. Consider the curves $q_i := \sigma(\partial \bar{V}_i)$ for $i=1,\ldots, n$. These curves lie on the tori $\partial V_i$ and, since they are sections, they intersect each principal orbit on $\partial V_i$ exactly once, so $\I{h_i}{q_i} = \pm 1$.

We orient the $q_i$ so that $\I{h_i}{q_i} = + 1$.
The $q_i$ are the boundary components of the image of the section, $\sigma\left(\bar{R}\right)$, but to achieve this sign convention they must be oriented by the outward normal to $\sigma\left(\bar{R}\right)$. That is, the orientation of the $q_i$ is that which agrees with $\partial \bar{V}_i$ as the boundary of $\bar{V}_i$.

The pair $\left( h_i, q_i \right)$ form a second basis of $H_1 (\partial V_i ; \zzz)$, with the same orientation, and so they can be related by some element of $\SL(2,\zzz)$,
$$ \begin{pmatrix}
h_i \\ q_i
\end{pmatrix}
=
\begin{pmatrix}
	\alpha_i & \gamma_i \\ -\beta_i & \delta_i
\end{pmatrix}
\begin{pmatrix}
l_i \\ m_i
\end{pmatrix}.$$

Since $q_i = -\beta_i l_i + \delta_i m_i$, we have that $\beta_i = \I{m_i}{q_i}$ is independent of our choice of longitude, and so we obtain a second Seifert invariant for the orbit. By choosing a different section, $\sigma'$, the class $q_i$ may be altered by adding a multiple of $h_i$,  although in this case at least some of the other $q_j$ for $j \neq i$ must also change. The invariant $\beta_i$ is therefore only determined up to a multiple of $\alpha_i$. 

Since $\left(\begin{smallmatrix}\alpha_i & \gamma_i \\ -\beta_i & \delta_i\end{smallmatrix}\right) \in \SL(2,\zzz)$, we have that $\gcd(\alpha_i, \beta_i) = 1$. 
Furthermore, inverting the matrix, we can write $m_i = \alpha_i q_i + \beta_i h_i$.

The pair $(\alpha_i, \beta_i)$ are the Seifert invariants for the orbit $T^1 (x_i)$ and we say that the Seifert invariants of the $T^1$ action on $Z^3$ are $$\left\{g;  (\alpha_1, \beta_1), \ldots, (\alpha_n, \beta_n)  \right\}.$$

These invariants are not uniquely determined by the action, not only because the index set may be permuted, but also for a more significant reason, already noted above: the $\beta_i$ depend on the choice of section $\sigma$. However, the generalized Euler number given by the sum 
\begin{equation}\label{e:eulernumber}
e = - \sum_{i=1}^n \frac{\beta_i}{\alpha_i}
\end{equation}
is independent of the choice of $\sigma$. 

In Proposition \ref{p:loopandspur}, we  calculate the orbifold fundamental group of some $4$-spaces. A necessary ingredient in this calculation is the fundamental group of a Seifert manifold. The following theorem shows how the Seifert invariants may be used to compute the fundamental group of the manifold. See Theorem 6.1 in~\cite{JN} for the proof. We give the result below for the case $g=0$, since this is the only case we are interested in here.

\begin{theorem}\label{t:pi1}
	Let $Z^3$ be a Seifert manifold with invariants $$\left\{0;  (\alpha_1, \beta_1), \ldots, (\alpha_n, \beta_n)  \right\}.$$ Then the fundamental group is given by
	$$ \pi_1(Z) = \left \langle q_1, \ldots, q_n, h \st \left[ q_i, h \right] = 1, q_i^{\alpha_i}h^{\beta_i} = 1, q_1 q_2 \cdots q_n = 1 \right\rangle.$$
\end{theorem}

Note that Seifert invariants are usually normalized \cite{Orlik}. The most frequent convention is to impose the constraint $0 < \beta_i < \alpha_i$ and to include an additional principal orbit of type $(1,b)$, so that we still have $b + \sum_{i=1}^n \frac{\beta_i}{\alpha_i} = -e$. 

The most natural point of view to take in understanding the normalized invariants is that they arise from making the choice of the sections $q_i$ first in order to satisfy the constraints $0 < \beta_i < \alpha_i$. The obstruction to extending the $q_i$ to a global section $\sigma$ is then given by $b$.
However, in the present work it is convenient to fix the global section $\sigma$ first, and so unnormalized invariants are more suitable.

\begin{table}[htbp]\caption{Table of notation for unnormalized Seifert invariants}\label{table}
	\centering 
	\begin{tabular}{r p{10.7cm} }
		
		$Z$ & A $3$-manifold with a fixed-point-free $T^1$ action.\\
		$\bar{Z}$ & Its orbit space, $Z/T^1$, a surface with singular points.\\
		$g$ & The genus of the surface $\bar{Z}$.\\
		$\bar{x}_i$ & For $i=1,\ldots,n$, a collection of points in $\bar{Z}$ which includes the image of each exceptional orbit $T^1 \left( x_i \right)$. \\
		$\bar{V}_i$ & A $2$-disk neighborhood of $\bar{x}_i$. \\ 
		${V}_i$ & Its preimage, a solid torus neighborhood of $T^1 \left( {x}_i \right)$. \\
		$\bar{R}$ & The complement of the $\bar{V}_i$, $\cl \left( \bar{Z} \setminus \left( \bigcup_{i=1}^n \bar{V}_i \right) \right)$, a surface with boundary.\\
		$\sigma$ & A section of the action defined on $\bar{R}$.\\
		$S_i$ & The slice at some $x_i$, oriented so that $\I{S_i}{T^1(x_i)}=+1$.\\
		$m_i$ & The boundary of $S_i$, a meridian on $\partial V_i$.\\
		$l_i$ & A longitude on $\partial V_i$ chosen so that $\I{l_i}{m_i}=+1$.\\
		$h_i$ & An oriented principal orbit on $\partial V_i$.\\
		$q_i$ & The oriented curve $\sigma (\partial \bar{V}_i) \subset \partial V_i$, which depends on $\sigma$ and is a section for the action on $\partial V_i$ such that $\I{h_i}{q_i} = +1$.\\
		$\alpha_i$ & The order of the isotropy at $x_i$, given by $\I{h_i}{m_i}$.\\
		$\beta_i$ & The intersection number $\I{m_i}{q_i}$, which depends on the choice of $\sigma$.\\
		
	\end{tabular}
	\label{tab:notation}
\end{table}

\section{Singularities and positive curvature}\label{s:singularities}

\subsection{Extents of spaces of directions}

For a metric space $X$, we define the $q$-extent of $X$ to be the maximum of the average distance function among $q$ points in $X$:
$$ \xt_q(X)= \max_{\{x_1, x_2,\ldots ,x_q\}\subset X}  \frac{1}{\binom{q}{2}} \sum_{i<j}  \dist(x_i, x_j).$$

In particular, the following lemma, originally given for the case where $X$ is a quotient of a Riemannian manifold but stated below in full generality, provides a particularly nice application of extents to control the number of highly singular points in the presence of positive curvature.

\begin{lemma}[Extent Lemma~\cite{GS2, GM}]\label{extent}   Let $X$ be an Alexandrov space. For any choice of $(q + 1)$ distinct points $p_0, \ldots, p_q \in X$ one has
	$$\frac{1}{q+1} \sum_{i=0}^q \xt_q \left( \Sigma_{p_i} X \right) \genfrac{}{}{0 pt}{}{>}{(=)}\, \frac{\pi}{3}$$
	whenever one has $\curv(X) \genfrac{}{}{0 pt}{}{>}{(=)} 0$.
\end{lemma}

In light of this result, we make the following distinction between spaces of directions.

\begin{definition}[Small Spaces] Let $\Sigma$ be an Alexandrov space of $\curv \geq 1$. We say that $\Sigma$ is {\em small} if $\xt_3(\Sigma)\leq \tfrac{\pi}{3}$. Then for a general Alexandrov space $X$ and a point $p\in X$ we can say that  $p$ has a {\em small} space of directions if $\xt_3(\Sigma_p)\leq\tfrac{\pi}{3}$. 
\end{definition}

It then follows that in a positively curved Alexandrov space we can bound the total number of points with small spaces of directions. 

\begin{proposition}\label{p:eulerbound}
	Let $X$ be a positively curved Alexandrov space. Then there can be at most three distinct points with small spaces of directions. 
\end{proposition}

\begin{proof}
	Assume, aiming for a contradiction, that there are four such points. Then  there would be four points in $X$ having spaces of directions with $3$-extent bounded above by $\pi/3$. By a simple application of the Extent Lemma \ref{extent}, this is not possible. Therefore there are at most three such points.
\end{proof}

We remark that the application of the Extent Lemma in this proof essentially synthesizes the argument given in the original paper of Hsiang and Kleiner~\cite{HK} to bound the total number of isolated fixed points. 

By design, Condition \q{} ensures that the space of directions at the image in the orbit space of an isolated fixed point is small. Therefore Proposition \ref{p:eulerbound} yields the following bound on the number of fixed points.

\begin{proposition}\label{p:fixbound}
	Let $X$ be a positively curved Alexandrov space on which $T^1$ acts isometrically and effectively so as to satisfy Condition \q{}. Then there can be at most three isolated fixed points. 
\end{proposition}

\subsection{Singular knots}\label{s:knots}

Under Condition \q{}, not only is the space of directions at the image in the orbit space of an isolated fixed point a small space, so too is any double  branched cover over two points corresponding to finite isotropy. This allows us to investigate singular knots by combining the results of the previous section with results on double branched covers over knots.

We define the following  condition.

\begin{definition}[Condition O]\label{d:condition_o} Let $X^3$ be an Alexandrov space homeomorphic to $\sph^3$ containing an extremal closed curve, $c$. We say that $(X,c)$ satisfies 
	{\em Condition O} if  there are two points $p_1, p_2\in X$,  not necessarily on $c$, which have small spaces of directions  and that the corresponding points $\tilde{p}_i, i=1, 2$ in the double branched cover of $X$ over $c$, $X_2(c)$, also have small spaces of directions. 
\end{definition} 

The following theorem synthesizes information from Section 2 of~\cite{GW2}.

\begin{theorem}\label{knot} Let $X^3$ be a positively curved Alexandrov space homeomorphic to $\sph^3$. If $c\subset X^3$ is an extremal closed curve and $(X^3, c)$ satisfies Condition O, then $c$ is unknotted.
\end{theorem}

\begin{proof} Consider the double branched cover over the closed curve $c$, $X^3_2(c)$.  By Condition O, we see that $X^3_2(c)$ has at least two singular points with small spaces of directions. Recall that by Lemma 5.2 in~\cite{GW2}, $X^3_2(c)$ is also positively curved. It follows that its universal cover, $\widetilde{X^3_2}(c)$, is also positively curved, with  at least $2|\pi_1(X^3_2(c))|$ singular points. By Proposition  \ref{p:eulerbound}, $2|\pi_1(X^3_2(c))| \leq 3$, 
	hence $X^3_2(c)$ is simply connected.
	
	  Theorem C of~\cite{GW2}  tells us that we can determine whether a closed curve, $c$, is knotted by considering the double branched cover, $X^3_2(c)$, over that curve. Namely, if $c$ is the unknot, $X^3_2(c)$ is simply connected, and otherwise the fundamental group has order at least $3$. In particular, this allows us to conclude that $c$ is the unknot.
\end{proof}

\subsection{Equivariant Suspension Theorem}

If the orbit space of a $G$-action is given by a suspension, with the suspension points corresponding to $G$-fixed points, then the $G$-action is also on a suspension. This allows us to give an Alexandrov geometry analogue of Grove and Searle's Equivariant Sphere Theorem~\cite[Theorem 1.4]{GS2}.

\begin{proposition}
	\label{p:suspension}
		Let $X$ be a closed, positively curved Alexandrov space on which $G$ acts by isometries. Suppose that $p_1, p_2 \in X$ are two fixed points and that $\bar{X} = X/G$ is homeomorphic to a suspension so that the homeomorphism respects the stratification by orbit types while taking $\bar{p}_1$ and $\bar{p}_2$ to the suspension points.

		Then $\Sigma_{p_1} \cong \Sigma_{p_2}\cong \Sigma$ are equivariantly homeomorphic, and $X$ is equivariantly homeomorphic to the spherical suspension of $\Sigma$.
\end{proposition}

\begin{proof}
	Following the methodology of Section 4 of ~\cite{HS}, let $W$ be the ``blow up" of $X$ at the fixed points $p_1$ and $p_2$. That is,  remove the points $p_1$ and $p_2$ from $X$ and replace them with their spaces of directions to obtain $W$. Then $W$ admits a $G$-action and the orbit space $W/G$ corresponds to the blow up of $\bar{X}$ at $\bar{p}_1$ and $\bar{p}_2$. 
	
	Since $\bar{X}$ is a suspension, $W/G$ is homeomorphic to $\bar{\Sigma} \times [0,1]$, where the orbit types respect the product structure and $\bar{\Sigma} \cong \Sigma_{\bar{p}_1} \cong \Sigma_{\bar{p}_2}$. By the Covering Homotopy Theorem of Palais~\cite{Pa} (cf.~Theorem II.7.1 in~\cite{Br}), $W$ is also a product and $G$ acts on it with a product action. 
	
	It follows that $\Sigma_{p_1}$ and $\Sigma_{p_2}$ are equivariantly homeomorphic, and so we identify them as $\Sigma$. Then $X$ is a suspension on $\Sigma$, and the group action on $X$ is the suspension of the isotropy action on $\Sigma$.
\end{proof}

\begin{theorem}[Equivariant Suspension Theorem]\label{t:est}
	Let $X$ be a closed, positively curved Alexandrov space on which $G$ acts by isometries. Suppose that $p_1, p_2 \in X$ are two fixed points and that $\diam \left( \Sigma_{\bar{p}_i} \right) \leq \frac{\pi}{4}$ for $i=1,2$. 
	Then $\Sigma_{p_1} \cong \Sigma_{p_2}\cong \Sigma$ are equivariantly homeomorphic, and $X$ is equivariantly homeomorphic to the spherical suspension of $\Sigma$.
\end{theorem}

\begin{proof}
	Denote by $\bar{X}$ the orbit space $X/G$. Let $q \in \bar{X} \setminus \left\lbrace \bar{p}_1, \bar{p}_2 \right\rbrace$ be chosen arbitrarily. Then $\angle \bar{p}_1 q \bar{p}_2 > \frac{\pi}{2}$. Then, just as in Theorem 4.5 in~\cite{P}, the function $\dist(p_i, \cdot)$ is regular, so that $\bar{X} \setminus \left\lbrace \bar{p}_1, \bar{p}_2 \right\rbrace$ fibers over an interval, and $\bar{X}$ is homeomorphic to the suspension of $\bar{\Sigma}$, where $\Sigma \cong \Sigma_{\bar{p}_1} \cong \Sigma_{\bar{p}_2}$. Following the method of proof of the Relative Stability Theorem 4.3 of the authors' previous work~\cite{HS}, this homeomorphism respects the stratification by extremal subsets. 
	
	The result then follows from Proposition \ref{p:suspension}.
\end{proof}

In Section \ref{s:twofixedpoints}, we  apply this result in the context of $T^1$ actions on $4$-dimensional spaces satisfying Condition \q{} in order to understand the isotopy type of extremal $\uptheta$-graphs, that is, graphs with two vertices, three edges and no loops, with the shape of the letter $\uptheta$. Where the isotropy action at a fixed point $p$ has three components of finite isotropy, Condition \qprime{} at that point implies that $\diam \left( \Sigma_{\bar{p}} \right) \leq \frac{\pi}{4}$, and so the result is applicable.

Alternatively, arguments similar to those used in Section \ref{s:knots} for extremal knots together with Lemma 2.3 of Calcut and Metcalf-Burton~\cite{CMB} could be applied to show that the $\uptheta$-graphs are unknotted.

\section{Topological classification for three isolated fixed points}\label{s:top3points}

In this section we  classify actions of $T^1$ on closed oriented Alexandrov spaces, $X^4$, such that the orbit space $\bar{X}^3 \cong \sph^3$ and the fixed-point set comprises three isolated points, $\left\lbrace p_1, p_2, p_3 \right\rbrace$, with the entire singular set lying on a locally flat closed curve $\bar{c}$, which is unknotted. The classification is up to orientation-preserving equivariant homeomorphism.

The classification we obtain is based on Fintushel's classification of simply connected $4$-manifolds with circle actions~\cite{F1}. We generalize his work, in the sense that we consider Alexandrov spaces rather than manifolds, but we also specialize it by constraining the structure of the orbit space.

In Section \ref{s:invariants} we assign invariants, $$\left( \frac{\beta_1}{\alpha_1}, \frac{\beta_2}{\alpha_2}, \frac{\beta_3}{\alpha_3}\right) \in \qqq^3,$$  to the actions in a manner similar to that of~\cite{F1}. There are various degrees of freedom in how these invariants are defined, and in Section \ref{s:equivalence} we define an equivalence relation on $\qqq^3$ which removes these freedoms. The invariants are then uniquely defined up to equivalence. In Lemma \ref{inv}, we  also show that certain invariants do not correspond to actions: those where  
$\frac{\beta_i}{\alpha_i}= \frac{\beta_j}{\alpha_j}$ for some $i\neq j$.  In Section \ref{s:uniqueness}, we prove a uniqueness result, showing that if two $4$-spaces have equivalent invariants there is an orientation-preserving equivariant homeomorphism between them, so that these invariants are sufficient for classification.

In Section \ref{s:cp2}, we prove an existence result which is  a key ingredient in proving Theorem \ref{t:amsr}, our main result. Theorem \ref{t:wcp} shows that every set of invariants, barring those where two of the triple are equal, corresponds to an action on a finite quotient of some $\wcp$.

\begin{remark}\label{generalequivalence}
All of the results in Sections 4.1, 4.2 and 4.3  can easily be shown to hold for any number $n$ of isolated points lying on an unknotted curve with the appropriate modifications. Note that those $n$-tuples for which  
$\frac{\beta_i}{\alpha_i}= \frac{\beta_{i+1}}{\alpha_{i+1}}$, where the indices are taken modulo $n$, do not correspond to an action.  
\end{remark}

\subsection{Assigning invariants to actions}\label{s:invariants}

In this subsection we assign invariants to actions on oriented $4$-spaces. Fintushel refers to these invariants as weights in~\cite{F1}, but to prevent ambiguity later when making use of weighted projective spaces, we  avoid this terminology.

Orient $\bar{X}$ so that its orientation, followed by the orientation of the orbits, gives the orientation of $X^4$.

Label the arc of $\bar{c}$ between the pair of fixed points $(\bar{p}_{i}, \bar{p}_{i+1})$ in $X^4 / T^1 = \bar{X}^3 \cong \sph^3$ by $J_i$, as shown in Figure \ref{f:orbitspace}, where the arithmetic here is modulo $3$. Orient $\bar{c}$ so that it is traversed as $(\bar{p}_1 \bar{p}_2 \bar{p}_3)$. Note that a renumbering of the $\bar{p}_i$ permits the opposite orientation of the curve.

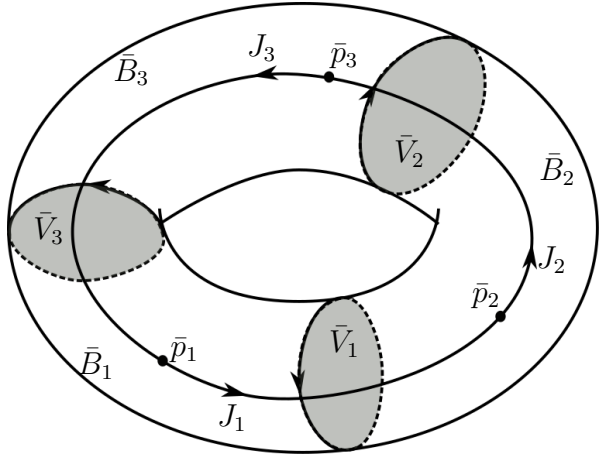
\begin{figure}
	\begin{tikzpicture}
	\node[anchor=south west,inner sep=0] (image) at (0,0) {\includegraphics[height=6cm]{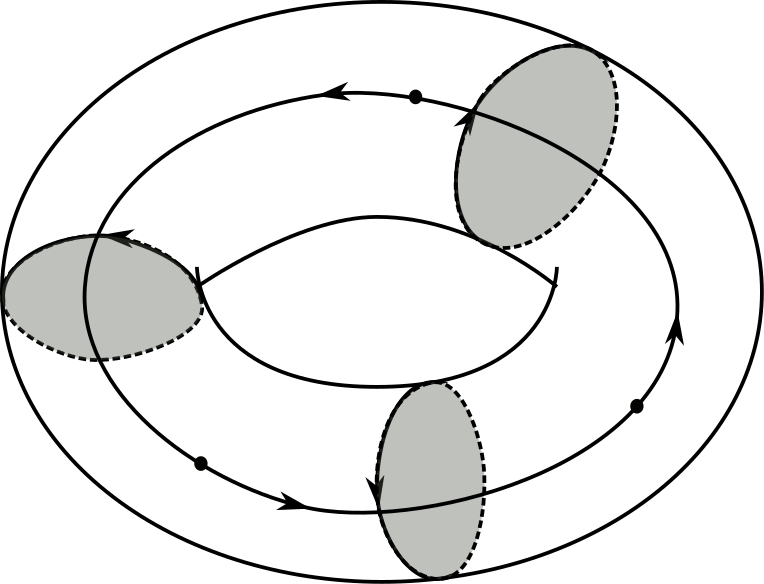}};
	\begin{scope}[
	x={(image.south east)},
	y={(image.north west)}
	]
	\node  at (0.57,0.25) {$\bar{V}_1$};
	\node  at (0.68,0.67) {$\bar{V}_2$};
	\node  at (0.07,0.5) {$\bar{V}_3$};
	\node  at (0.38,0.08) {${J}_1$};
	\node  at (0.92,0.43) {${J}_2$};
	\node  at (0.43,0.9) {${J}_3$};
	\node  at (0.3,0.23) {$\bar{p}_1$};
	\node  at (0.81,0.35) {$\bar{p}_2$};
	\node  at (0.57,0.88) {$\bar{p}_3$};
	\node  at (0.15,0.225) {$\bar{B}_1$};
	\node  at (0.93,0.63) {$\bar{B}_2$};
	\node  at (0.22,0.85) {$\bar{B}_3$};
	\end{scope}
	\end{tikzpicture}
	\caption{Structure of the neighborhood $U$ of the curve $\bar{c} \subset \bar{X}^3\cong S^3$.}
	\label{f:orbitspace}
\end{figure}

Let $\bar{U} \cong \sph^1 \times \disk^2$ be a closed neighborhood of $\bar{c}$. We decompose $\bar{U}$ into three balls $\bar{B}_i$, each containing one fixed point $\bar{p}_i$, so that $\bar{B}_i \cap \bar{B}_{i+1} = \bar{V}_i$, a locally flat $2$-disk. Then $\partial B_i$, a submanifold of $X^4$, is a Seifert manifold, which we  denote by $Z^3_i$, and we orient it by its inward normal.

Orient each disk $\bar{V}_i$ so that the intersection number $\I{\bar{V}_i}{J_i} = +1$. This is equivalent to orienting it as a submanifold of $\bar{Z}_{i+1}$ (see Figure \ref{f:seifert}). We then orient the boundary $\partial \bar{V}_i$ by the inward normal.

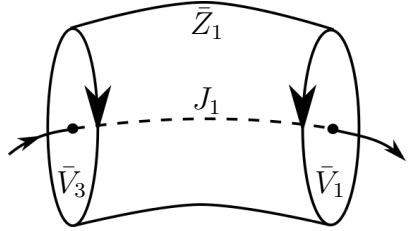
\begin{figure}
	\begin{tikzpicture}
	\node[anchor=south west,inner sep=0] (image) at (0,0) {\includegraphics[height=3cm]{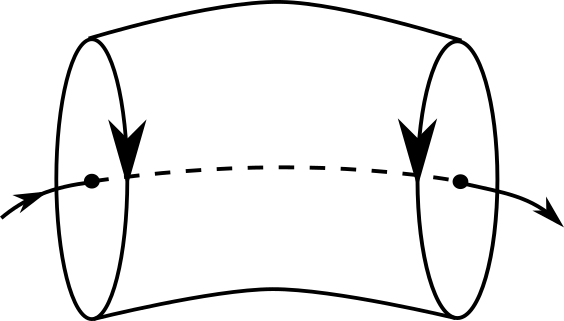}};
	\begin{scope}[
	x={(image.south east)},
	y={(image.north west)}
	]
	\node  at (0.5,0.9) {$\bar{Z}_1$};
	\node  at (0.81,0.2) {$\bar{V}_1$};
	\node  at (0.16,0.2) {$\bar{V}_3$};
	\node  at (0.5,0.56) {${J}_1$};
	\end{scope}
	\end{tikzpicture}
	\caption{The base of the Seifert manifold $Z_1 = \partial B_1$. Note that $\bar{V}_3$ is oriented as a submanifold of $\bar{Z}_1$, since $J_1$ corresponds to the inward normal here. However, $\bar{V}_1$ has the opposite orientation, and is instead oriented as a submanifold of $\bar{Z}_2$.}
	\label{f:seifert}
\end{figure}

We decompose $\bar{X}^3$ as $\bar{U} \cup \bar{W}$, where $\bar{W} = \cl \left(\bar{X} \setminus \bar{U} \right)$. Since $\bar{c}$ is unknotted, $\bar{W}$ is homeomorphic to $\sph^1 \times \disk^2$ and, in particular, $H^2(\bar{W};\zzz)=0$. Indeed, a simple application of the Mayer--Vietoris sequence shows that $H^2$ is trivial even when $\bar{c}$ is knotted. The free circle action on $W$ therefore corresponds to a trivial $T^1$-principal bundle, which admits a section, $\sigma \colon \bar{W} \to W$. 

Let the image of the restricted section $\left. \sigma \right|_{\partial \bar{V}_i}$, oriented by $\partial \bar{V}_i$, be denoted by $q_i$. We may use $q_i$ to define the pair of Seifert invariants $(\alpha_i, \beta_i)$ for the exceptional fiber in $V_i \to \bar{V}_i$, following the conventions given in Section \ref{s:seifert}. If $J_i$ does not correspond to a component of finite isotropy, then there is no exceptional fiber and the pair of invariants is of the form $(1, \beta_i)$.

\begin{definition}
	The invariants of the $T^1$ action on the oriented space $X^4$, with respect to the section $\sigma$ and the labeling of the curve $\bar{c}$, are $$\left( \frac{\beta_1}{\alpha_1}, \frac{\beta_2}{\alpha_2}, \frac{\beta_3}{\alpha_3}\right) \in \qqq^3.$$
\end{definition}

We  always assume that fractions are written in their lowest terms, with a positive denominator. Note that reversing the orientation of $X^4$  changes the sign of the invariants.

These invariants differ from Fintushel's weighted orbit spaces~\cite{F1} in three key ways, which we review here.

\begin{enumerate}[label=\Roman*., wide, labelwidth=!, labelindent=0pt, itemsep=1ex]
	\item \textbf{Normalization of invariants.} 
	Fintushel requires that the invariants for the exceptional fiber in $V_i \to \bar{V}_i$ be normalized so that $0 < \beta_i < \alpha_i$. These normalized invariants then correspond to a particular choice of section $\partial \bar{V}_i \to \partial {V}_i$. Theorem 3.6 of~\cite{F1} shows that these may be extended to a section $\partial \bar{U} \to \partial U$.
	
	However, the existence of an extension relies in a key way on the spaces of directions, $\Sigma _{p_i}$, being spheres, which does not hold in Alexandrov geometry. In our case, the normalized invariants need not extend to a section over $\partial \bar{U}$. For this reason, we begin with the global section, and so use unnormalized invariants.
	
	\item \textbf{Virtual edges.} If two of the three fixed points are not joined by a component of finite isotropy, the singular set in the orbit space does not contain a closed curve. Fintushel analyzes the resulting arcs and isolated fixed points separately. However, given the limited number of orbit space types being considered here, it is simplest to unify the treatment. 
	
	The singular set does lie on a closed curve, $\bar{c}$, though some arcs $J_i$ on the curve correspond to principal orbits. These arcs are ``virtual edges'' in the orbit space. They are treated as though they correspond to orbits with finite isotropy of order $1$. This results in the inclusion of invariants of the form $(1,\beta_i)$ in the set of invariants for $\Sigma_{p_i}$. 
	
	\item \textbf{Weighted orbit spaces.} Finally, Fintushel attaches his invariants to the orbit space to create a weighted orbit space. The isomorphism type of this weighted orbit space is a key concept in his arguments, and is a convenient way of keeping track of the different links and arcs in the space. In our case, since an orbit space always comprises a single unknotted curve in the sphere, we suppress any explicit mention of isomorphisms of weighted orbit spaces. The concept can be reduced to the equivalence of invariants, and this is done in the next section.
\end{enumerate}

\subsection{Equivalence of invariants}\label{s:equivalence}

Observe that the section $\sigma \colon \bar{W} \to W$ is not unique, and hence the invariants of the $T^1$ action on $X^4$ are not uniquely defined. Such a section is a map $\sph^1 \times \disk^2 \to \T^2 \times \disk^2$. Consider a longitudinal circle, $\sph^1 \times \left\lbrace x \right\rbrace$ and lift it to $\T^2 \times \disk^2$. The homology class of the lifted circle clearly does not depend on $x$, and it can be varied by adding $kh$, where $k \in \zzz$ and $h \in H_1(W; \zzz)$ is an oriented orbit. Varying the section in this way changes each of the $\beta_i$ to $\beta_i + k \alpha_i$, adding $k$ to each fraction.

Now we consider the impact of this change on the invariants for the Seifert manifold $Z_i \to \bar{Z}_i$, oriented by its inward normal. There are two exceptional fibers, and they have neighborhoods $\bar{V}_{i-1}$ and $-\bar{V}_i$, by which we mean $\bar{V}_i$ with the reverse orientation (see Figure \ref{f:seifert}). 
With respect to the original section $\sigma$, the invariants are therefore $\left\{0;  (\alpha_{i-1}, \beta_{i-1}), (\alpha_i, -\beta_i)  \right\}$. Varying $\sigma$, we see that $\beta_{i-1}$ is increased by $k \alpha_{i-1}$ while $- \beta_i$ is decreased by $k \alpha_{i}$, so that the generalized Euler number, given by the sum of the fractions as in Equation (\ref{e:eulernumber}), is unchanged as required.

Moreover, the labeling of the fixed points on $\bar{c}$ is also arbitrary. The starting point of the closed curve, as well as its orientation, may be freely changed. These changes also alter the invariants.

The effects of these choices on the invariants can be expressed using the following equivalence relation on $\qqq^3$.

\begin{definition}
	{Two points in $\qqq^3$ are related by the equivalence relation $\sim$ if they are the same up to
		\begin{enumerate}
			\item Rotation of co-ordinates: $(a,b,c) \sim (b,c,a)$;
			\item Reversal of co-ordinates with a change of sign:\\ $(a,b,c) \sim (-c,-b,-a)$;
			\item Translation by an integer multiple of $(1,1,1)$:\\ $(a,b,c) \sim (a+k, b+k, c+k)$ for some $k \in \zzz$;
		\end{enumerate}
		or any combination of (1), (2) and (3).  }
\end{definition} 

The following lemma is then clear from the preceding discussion.

\begin{lemma}\label{l:equivalence}
	Let $T^1$ act on $X^4$, a closed oriented Alexandrov space, so that $\bar{X}^3 \cong \sph^3$ and the fixed-point set comprises three isolated points with the entire singular set lying on a closed curve which is unknotted. If the action has invariants $\left( a,b,c \right) \in \qqq^3$ with respect to some section $\sigma$ and some labeling of the curve, then for any $\left( a',b',c' \right) \sim \left( a,b,c \right) \in \qqq^3$ there is a section $\sigma'$ and labeling of the curve with respect to which the invariants of the action are $\left( a',b',c' \right)$.
\end{lemma}

Comparing once more to~\cite{F1}, we note that Relations (1) and (2), which are associated to the labeling of the curve, are referred to in~\cite[Section 3.3]{F1}, but Relation (3), which is associated to the section, is not used, because of the decision to use normalized invariants.

From now on, we  simply refer to the triple $\left( \frac{\beta_1}{\alpha_1}, \frac{\beta_2}{\alpha_2}, \frac{\beta_3}{\alpha_3}\right) \in \qqq^3$ as being the invariants of the action, and only make reference to the section and labeling of the curve where it is necessary.

It is not the case that any triple of fractions can correspond to an action. The next lemma gives a constraint on the values the fractions can take, and we  see later in Theorem \ref{t:wcp} that it is, in fact, the only constraint. This result may be compared to~\cite[Lemma 3.5]{F1}, in the more rigid manifold situation.

\begin{lemma}\label{inv}
	Let $T^1$ act on $X^4$, a $4$-dimensional, closed, oriented Alexandrov space, so that $\bar{X}^3 \cong \sph^3$ and the fixed-point set comprises three isolated points with the entire singular set lying on a closed curve which is unknotted. Let the action have invariants $\left( \frac{\beta_1}{\alpha_1}, \frac{\beta_2}{\alpha_2}, \frac{\beta_3}{\alpha_3}\right) \in \qqq^3$. Then the fractions $\frac{\beta_i}{\alpha_i}$, for $i=1,2,3$, are pairwise unequal.
\end{lemma}

\begin{proof}

	Recall that the Seifert invariants of the $T^1$ action restricted to the manifold $Z_i$ are $\left\lbrace 0; (\alpha_{i-1}, \beta_{i-1}), (\alpha_i, -\beta_i)\right\rbrace$. It therefore follows from the Orlik--Raymond classification of circle actions on $3$-manifolds that $Z_i$ is a lens space~\cite[Theorem 4]{OR}, unless $\frac{\beta_{i-1}}{\alpha_{i-1}} = \frac{\beta_{i}}{\alpha_{i}}$, in which case $Z_i \cong \sph^2 \times \sph^1$. However, $Z_i$ is homeomorphic to the space of directions at $p_i$, and so must admit positive curvature, and this yields a contradiction. It follows that all three fractions $\frac{\beta_i}{\alpha_i}$ are pairwise unequal.
\end{proof}

\subsection{Classification by invariants}\label{s:uniqueness}

In this subsection, we show that if two spaces admit circle actions with equivalent invariants then there is an orientation-preserving equivariant homeomorphism between them, so that these invariants suffice for classification. We then show, in the following subsection, that for any possible set of invariants, there is a finite quotient of a weighted complex projective space with a circle action with those invariants.

Before proceeding, we need a basic lemma on equivariant isotopies of the torus.

\begin{lemma}\label{l:isotopy}
	Let $T^1$ act freely on the torus $\T^2$. Let $\sigma, \sigma' \colon \sph^1 \to \T^2$ be two homologous sections of the action, and let $f \colon \T^2 \to \T^2$ be an equivariant homeomorphism which satisfies $f \circ \sigma = \sigma'$. Then there is an orbit-preserving equivariant isotopy from $\id_{\T^2}$ to $f$, that is, an isotopy through equivariant homeomorphisms all of which induce 
	$\id_{\sph^1}$.
\end{lemma}

\begin{proof}

Note first that $f$ is unique, since if $g$ were another such homeomorphism then $g^{-1} \circ f$ would be an equivariant homeomorphism fixing all points on the section $\sigma$, and therefore $g^{-1} \circ f = \id_{\T^2}$.
Let $\pi \colon T^2 \to \sph^1$ be the projection to the orbit space.
Write $f$ as $p \mapsto \left( \phi \circ \pi (p) \right) \cdot p$, where $\phi \colon \sph^1 \to T^1$ is a continuous map from the orbit space into the circle group.

For any two sections $\sigma$ and $\sigma'$ there is some $k \in \zzz$ so that, in integral homology, $[\sigma'] - [\sigma] = k[h]$, where $[h]$ is the homology class of the orbit (see Section \ref{s:seifert}).
Since $f \circ \sigma = \sigma'$, we have $\deg (\phi) = k$.
By assumption $\sigma$ and $\sigma'$ are homologous, so that $\deg (\phi) = 0$.
It follows that $\phi$ is homotopic to a constant map and, in particular, to the map to the identity element of $T^1$. This induces the desired isotopy from $\id_{\T^2}$ to $f$.
\end{proof}

The following theorem generalizes Theorems 3.6 and 6.2 of \cite{F1} to the Alexandrov space setting.

\begin{theorem}[Uniqueness]\label{t:classification}
	Let $X_1^4$ and $X_2^4$ be two $4$-dimensional, closed, oriented Alexandrov spaces, each admitting isometric circle actions. Suppose that these actions are such that $\bar{X_j}^3 \cong \sph^3$ for $j=1,2$ and that for each $j$ the fixed-point set comprises three isolated points with the entire singular set lying on a locally flat unknot. Suppose further that the actions have equivalent invariants. Then there is an orientation-preserving equivariant homeomorphism $X_1 \cong X_2$.
\end{theorem}

\begin{proof}

	Recall that invariants are defined with respect to a section and a labeling of the curve. By Lemma \ref{l:equivalence}, we may choose the sections and labelings so that the two actions have invariants which are equal, rather than merely equivalent. 
	
	Denote by $p_{i,1}$, for $i=1,2,3$, the three fixed points in $X_1^4$. Denote the fixed points in $X_2^4$ by $p_{i,2}$. The $X_j$ induce orientations of the orbit spaces $\bar{X}_j$, and the orbit spaces can be identified via an orientation-preserving homeomorphism which carries $\bar{p}_{i,1} \mapsto \bar{p}_{i,2}$ for each $i$. From here on we  denote both $\bar{X}_1 \cong \bar{X}_2$ by $\bar{X}$. Let $\bar{c}$ be the unknotted closed curve, let $\bar{U}$ be a regular neighborhood of $\bar{c}$, and let $\bar{U}$ be decomposed as described in Section \ref{s:invariants} (see Figure \ref{f:orbitspace}). 
	
	Let $\bar{W}=\cl(\bar{X} \setminus \bar{U})$  and let $\sigma_j \colon \bar{W} \to W_j \subset X_j^4$ be the sections, which have been chosen so that the invariants are equal.
	Using $\sigma_1$ and $\sigma_2$, we construct an orientation-preserving equivariant homeomorphism $f_W \colon W_1 \to W_2$ with $\bar{f}_{W} = \left.\id\right|_{\bar{W}}$ satisfying $f_W \circ \sigma_1 = \sigma_2$.
	
	Now consider two solid tori, $V_{i,1} \subset U_1$ and $V_{i,2} \subset U_2$, both over the disk $\bar{V}_i \subset \bar{U} \subset \bar{X}$, corresponding to the pair of invariants $(\alpha_i, \beta_i)$. Let $\left(\begin{smallmatrix}\alpha_i & \gamma_i \\ -\beta_i & \delta_i\end{smallmatrix}\right) \in \SL(2,\zzz)$.
	 Both $V_{i,j}$ are equivariantly homeomorphic to $\sph^1 \times \disk^2$ with a circle action given by $\mu \cdot ({\rm e}^{{\rm i} \theta}, r {\rm e}^{{\rm i} \phi}) = (\mu^{\alpha_i} {\rm e}^{{\rm i} \theta}, \mu^{\gamma_i} r {\rm e}^{{\rm i} \phi})$.
	 The choice of element in $\SL(2,\zzz)$ corresponds to specifying longitudinal circles on each $\d V_{i,j}$. 
		The solid tori $V_{i,1}$ and $V_{i,2}$ are therefore equivariantly homeomorphic to each other, and we choose $f_i \colon V_{i,1} \to V_{i,2}$, an equivariant homeomorphism preserving the specified longitudes, so that $\bar{f}_i = \left. \id \right|_{\bar{V}_i}$.

	Now consider the oriented curves $\sigma_j \left( \d \bar{V}_{i} \right)$. In $\d \left( \sph^1 \times \disk^2 \right)$ these curves are both homologous to $\nu \mapsto \left( \nu^{-\beta_i}, \nu^{\delta_i} \right)$. 
Therefore the equivariant homeomorphism $f_i$ is such that $\left. f_i \circ \sigma_1 \right|_{\d \bar{V}_{i}} := \sigma'_2$ is a new section homologous to $\left. \sigma_2 \right|_{\d \bar{V}_{i}}$.

We now let $h = \left. f_i \circ f_W^{-1} \right|_{\partial V_{i,2}}$ be the induced self-homeomorphism of $\partial V_{i,2}$.
By definition $h \circ \sigma_2 =  \sigma'_2$, and this new section $\sigma'_2$ is homologous to $\sigma_2$.
Therefore we can apply Lemma \ref{l:isotopy} to obtain an orbit-preserving equivariant isotopy from $h$ to $\left.\id\right|_{\partial V_{i,2}}$.
Composing, we obtain an isotopy from $\left. h \circ f_W \right| _{\partial V_{i,1}} = \left. f_i \right| _{\partial V_{i,1}}$ to $\left. f_W \right| _{\partial V_{i,1}}$. 
 
Since $f_i$ is defined over all of $V_{i,1}$, this allows us to extend $f_W$ over all of $V_{i,1}$ as well. 	We can then extend the homeomorphism over the interior of the $B_i$  by coning, which gives us the result.
 \end{proof}

As previously noted in Remark \ref{generalequivalence}, this result can be generalized to classify actions with any number of isolated fixed points, $n$, provided the entire singular set lies on an unknotted closed curve in $\sph^3$.

\subsection{The weighted complex projective space \texorpdfstring{$\wcp$}{CP\textasciicircum 2\_\{a,b,c\}}}\label{s:cp2}

Now, for any possible set of invariants, we construct a finite quotient of a weighted complex projective space with a circle action with those invariants.

The circle actions constructed on weighted complex projective spaces are induced by $T^3$ actions on $\sph^5$. We first describe the basic set-up.

When $T^3$ acts on $\sph^5$, there are three distinguished circle subgroups, up to orientation, acting fixed-point-homogeneously, and we  take these circles to be the generators of $H_1(T^3; \zzz)$. 

The quotient space $\sph^5 / T^3$ is a $2$-simplex, $\Delta^2$. The isotropy groups at the edges are the circle subgroups of $T^3$ that act fixed-point-homogeneously on $\sph^5$ and the isotropy groups at the vertices are the $T^2$ subgroups generated by pairs of these.

Let $a,b,c$ be integers satisfying $\gcd(a,b,c) = 1$. Define the subgroup $T^1_{a,b,c} \hookrightarrow T^3$ by the inclusion $\lambda \mapsto (\lambda^a, \lambda^b, \lambda^c)$. This group acts almost freely on $\sph^5$, and the quotient space $\sph^5 / T^1_{a,b,c}$ is the weighted complex projective space $\wcp$. It has an induced action by  
$T^2 = T^3 /T^1_{a,b,c}$. Note that by changing the signs of the weights, there are eight different ways to represent each $\wcp$, four in each orientation.

Since $T^1_{a,b,c}$ acts almost freely, the isotropy subgroups of the $T^2$ action on $\wcp$ are isomorphic to the isotropy subgroups of the $T^3$ action of $\sph^5$.
In particular, there are three fixed points of the $T^2$ action and there are three distinct circle subgroups of $T^2$, each acting fixed-point-homogeneously on $\wcp$. It follows that any other circle subgroup of $T^2$ has exactly three fixed points. We denote the homology classes in $H_1(T^2; \zzz)$ of the fixed-point-homogeneous circle subgroups  by $m_1$, $m_2$ and $m_3$.

Choose a basis for $T^2$ so that none of the fixed-point-homogeneous circles are basis elements. Consider the circle subgroup $T^1 \subgp T^2$ given by $\mu \mapsto (1,\mu)$. Denote its corresponding homology class by $h \in H^1(T^2; \zzz)$, so that $h$ is a basis element of $H_1 (T^2 ; \zzz)$ and $h \neq \pm m_i$ for any $i$. 
The numbers $\alpha_i = \I{h}{m_i}$, provide the order of the finite isotropy groups. Since the choice of orientation of the $m_i$ was arbitrary, we may assume $\alpha_i > 0$.

The orbit space $\wcp / T^1$ is homeomorphic to $\sph^3$. The $T^3$ action on $\sph^5$ induces one final circle action on $\wcp / T^1$. As observed above, the orbit space is homeomorphic to the $2$-simplex $\Delta^2$. There is an unknotted closed curve in $\sph^3$ which contains $\bar{F} \cup \bar{E}$, and this curve is fixed by the circle action. The image of the unknotted curve is $\partial \Delta ^2$.

\begin{theorem}[Existence]\label{t:wcp}
	Consider a set of invariants $$\left( \frac{\beta_1}{\alpha_1}, \frac{\beta_2}{\alpha_2}, \frac{\beta_3}{\alpha_3}\right) \in \qqq^3$$ so that the fractions $\frac{\beta_i}{\alpha_i}$, for $i=1,2,3$, are pairwise unequal. Then there is a weighted complex projective space $\wcp$ and a finite group $\Gamma$ so that the standard $T^2$ action on $\wcp / \Gamma$ contains a circle action with these invariants. 
\end{theorem}

\begin{proof}
	Consider an arbitrary weighted complex projective space, $\wcp$. As noted above, $\wcp$ admits an action by $T^2$, with orbit space a $2$-simplex $\Delta^2$, and three fixed-point-homogeneous circle subgroups represented in homology by $\pm m_i$. 
	
	Let $T^1$ be the circle subgroup of $T^2$ defined by $\mu \mapsto (1,\mu)$, and denote its homology class by $h$. By a suitable choice of basis for $T^2$, we may assume $h \neq \pm m_i$.  Denote the homology class of the circle subgroup defined by $\nu \mapsto (\nu, 1)$ by $q$. 
	The basis chosen for $T^2$ determines an orientation on each orbit so that $\I{q}{h} = +1$.

	Orient $\Delta^2$ so that its orientation followed by the orientation of the $T^2$ orbits gives the orientation of $\wcp$, and orient $\partial \Delta^2$ by its inward normal. Consider a collar neighborhood $Y$ of $\partial \Delta^2$. Fix three points $x_i$ on $\partial \Delta^2$, one on each edge, so that $\partial \Delta^2$ with its orientation is traversed as $(x_1 x_2 x_3)$.
	Draw curves $c_i$ starting from each $x_i$ and ending at $y_i \in \partial Y$ so that the $c_i$ split $Y$ into three components as shown in Figure \ref{f:delta}.
	
\begin{figure}
	\begin{tikzpicture}
	\node[anchor=south west,inner sep=0] (image) at (0,0) {\includegraphics[height=5cm]{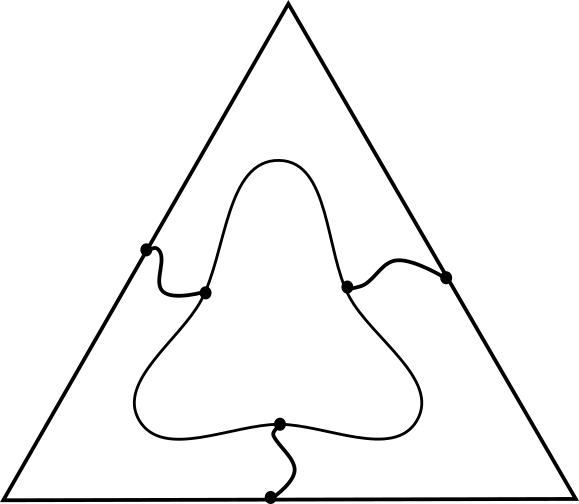}};
	\begin{scope}[
	x={(image.south east)},
	y={(image.north west)}
	]
	\node  at (0.5,-0.06) {$x_1$};
	\node  at (0.82,0.48) {$x_2$};
	\node  at (0.2,0.54) {$x_3$};
	\node  at (0.5,0.21) {$y_1$};
	\node  at (0.55,0.44) {$y_2$};
	\node  at (0.41,0.41) {$y_3$};
	\node  at (0.55,0.08) {$c_1$};
	\node  at (0.67,0.53) {$c_2$};
	\node  at (0.28,0.385) {$c_3$};
	\node  at (0.12,0.085) {$Y$};
	\end{scope}
	\end{tikzpicture}
	\caption{The structure of $\wcp / T^2$, which is the quotient space of $\bar{X}^3$, as depicted in Figure \ref{f:orbitspace}, by a further circle action.}
	\label{f:delta}
\end{figure}
	
	The preimage  in $\wcp / T^1 \cong \sph^3$ of this structure on $Y$ consists of an unknot, $\bar{c}$, with a tubular neighborhood $\bar{U}$, and three separating disks, denoted by $\bar{V}_i$, as depicted in Figure \ref{f:orbitspace}. We can use these to define the invariants of the action. 
	
	Orient $\sph^3$ by the orientation of $\Delta^2$, followed by the orientation of the orbit of the induced circle action on $\sph^3$. Our choice of $T^1 \subgp T^2$ as the second basis element guarantees that this orientation of $\sph^3$, followed by the orientation of the orbit of $T^1$ on $\wcp$, is the orientation of $\wcp$, as required. 
	
	Orient the unknot $\bar{c}$ by $\partial \Delta^2$. Then if $\partial \bar{V}_i$ is given its orientation as the circle orbit over $y_i$, and $\bar{V}_i$ is oriented by $\partial \bar{V}_i$ with respect to the inward normal, $\I{\bar{V}_i}{\bar{c}} = +1$. To see this, note that the inward normal to $\partial \bar{V}_i$ descends to a vector which points in the opposite direction to the inward normal to $\partial \Delta^2$.

	The principal portion of the $T^2$ action on $\wcp$ is a trivial principal bundle over the interior of $\Delta^2$ and so admits a section, $\tau$. 
	We can use the circle subgroup of $T^2$ given by $\nu \mapsto (\nu, 1)$, whose homology class is denoted by $q$, together with the section $\tau$ to define a section of the $T^1$ action, $\sigma \colon \left( \sph^3 \setminus \bar{c} \right) \to \wcp$.

	To use $\sigma$ to define the invariants, we must identify the oriented curve $\sigma \left( \partial \bar{V}_i \right) \subset \d V_i$. Since $\d \bar{V}_i$ is oriented as the circle orbit over $y_i$, $\sigma \left( \partial \bar{V}_i \right) = q$. Recall, however, that when $\d V_i$ is considered as the $T^2$ orbit over $y_i$, we have $\I{h}{q} = -1$, which has the wrong sign for calculating invariants. This is resolved by observing that the orientation of $\d V_i$ as a boundary is, in fact, opposite to its orientation as an orbit.
	
	We also need to identify the meridian curves. Since the circle subgroup corresponding to the curve $m_i$ fixes the core of the solid torus $V_i$, $m_i$ is contractible in $V_i$ and so is a meridian.
	
	The intersection number $\I{h}{m_i} = \alpha_i$, the order of the finite isotropy, while $\I{q}{m_i} = - \beta_i$. Therefore, we need $m_i = \alpha_i q + \beta_i h$.
	
	The map $p \colon T^3 \to T^3 / T^1_{a,b,c}$ induces a map $p_* \colon H_1 (T^3 ; \zzz) \to H_1 (T^2 ; \zzz)$ in homology. Clearly $p$  brings the fixed-point-homogeneous circles in $T^3$ to the desired $m_i \in H_1 \left( T^2 ; \zzz \right)$
	if $p_* = \left(\begin{smallmatrix}\alpha_1&\alpha_2&\alpha_3\\\beta_1&\beta_2&\beta_3\end{smallmatrix}\right)$. 
	Note that $p_*$ is of full rank. If one row were a multiple of the other, then $\frac{\beta_i}{\alpha_i}$ would have the same value for each $i$, but each fraction must be different.
	
	In case $\bar{\alpha} = \gcd(\alpha_i)$ and $\bar{\beta} = \gcd(\beta_i)$ are both $1$, to realize this map, it is enough to find the kernel of $p_*$ as a linear transformation $\rrr^3 \to \rrr^2$, and note that this corresponds to a circle subgroup of $T^3$, which should be taken to be $T^1_{a,b,c}$. The resulting $\wcp$ carries a $T^2$ action. The quotient by the subgroup $\zzz_{\bar{\alpha}} \times \zzz_{\bar{\beta}}$ can then be taken if necessary.
\end{proof}

\begin{remark}
	It might appear at first sight that not all $\wcp$ can be obtained in this way. For instance, any invariants associated to $\cp^2 = \cp^2_{1,1,1}$ must satisfy $\left(\begin{smallmatrix}\alpha_1&\alpha_2&\alpha_3\\\beta_1&\beta_2&\beta_3\end{smallmatrix}\right) \left(\begin{smallmatrix}1\\1\\1\end{smallmatrix}\right) = \left(\begin{smallmatrix}0\\0\end{smallmatrix}\right)$,
	which is not compatible with the constraint that $\alpha_i > 0$. However,  recalling that $\cp^2_{1,1,1} \cong \cp^2_{1,-1,-1}$, for example, the equation can be solved.
\end{remark}

\section{Almost maximal symmetry rank: dimension \texorpdfstring{$3$}{3}}\label{s:amsr3}

We consider circle actions on positively curved Alexandrov spaces of dimension three and four, beginning in this section with dimension three.

\begin{remark}
If a positively curved Alexandrov space has boundary, then there is a unique point at maximal distance from the boundary and the space is homeomorphic to a cone, with this ``soul'' point corresponding to the cone point. The soul point is fixed by the action, and by the Slice Theorem \ref{l:slice} the isotropy action at the soul determines the equivariant homeomorphism type (see~\cite{HS}). 
For this reason, we  always assume that the spaces have no boundary.  
\end{remark}

We recall first the manifold classification.
The topological classification of the underlying spaces follows from the work of Hamilton \cite{Ham}, which shows that the positively curved manifolds are precisely the spherical space forms. The actions on these manifolds were classified by Raymond~\cite{Ray}.

\begin{theorem}\label{p:amsrmanifold3} Let $T^1$ act isometrically and effectively on $M^3$, a closed, positively curved Riemannian manifold.
	Then  $M$ is equivariantly diffeomorphic to $\sph^3/\Gamma$ with a linear action, where $\Gamma$ is a freely acting finite subgroup of $\SO(4)$.
\end{theorem}

The extension of Theorem \ref{p:amsrmanifold3} to Alexandrov spaces in dimension $3$ is straightforward, and we  give it without the hypothesis of orientability.

\begin{proposition}[Circle actions in dimension $3$]\label{p:3space}
	Let $T^1$ act isometrically and effectively on $X^3$, where $X$ is a $3$-dimensional, closed, positively curved Alexandrov space. Then $X^3$  is equivariantly homeomorphic to $\sph^3/\Gamma$ with a linear $T^1$ action, where $\Gamma$ is a finite subgroup of $\O(4)$.
\end{proposition}

\begin{proof} It was shown in~\cite{GGG2} that every positively curved Alexandrov $3$-space, $X^3$, is homeomorphic to either a spherical manifold or  
$\Susp \left( \rp^2 \right)$. If it is a manifold, the result follows by Raymond \cite{Ray}. If it is $\Susp \left( \rp^2 \right)$, then it follows by the classification of circle actions on closed three-dimensional Alexandrov spaces \cite{NZ}.
\end{proof}

Note that not every $\Gamma < \O(4)$ yields an orbifold $\sph^3/\Gamma$ which admits a circle action: Dunbar~\cite{Dunbar} demonstrates that there are twenty-one subgroups of $S\O(4)$ yielding nonfibering orbifolds.

\section{Almost maximal symmetry rank: dimension \texorpdfstring{$4$}{4}}\label{s:amsr4}

We proceed to consider circle actions on positively curved Alexandrov spaces of dimension four.
As before, we  always assume that the space has no boundary.  We  provide the classification for orientable spaces only: the non-orientable spaces are quotients by an involution which commutes with the circle action by Theorem A of ~\cite{HS}. 

We recall first the manifold classification.

\begin{theorem}\label{p:amsrmanifold4}\cite{HK, GS1, GW2} Let $T^1$ act isometrically and effectively on $M^4$, a $4$-dimensional, closed, positively curved Riemannian manifold.
Then
$M$ is equivariantly diffeomorphic to $\sph^4$, $\rp^4$ or $\cp^2$ with a linear action.
\end{theorem}
As mentioned in the Introduction, the proof  for simply connected $4$-manifolds is given up to equivariant homeomorphism by~\cite{HK} and up to equivariant diffeomorphism by~\cite{GS1} and~\cite{GW2}. 
 Note that while all these actions extend to actions by homeomorphisms of $T^2$ (see~\cite{GW2}), in contrast, in dimension three, $\sph^3/\Gamma$ does not admit a $T^2$ action when $\Gamma$ is not a cyclic group by work of Mostert~\cite{M} and Neumann~\cite{ N}. Thus, we see that there are some spaces in Theorem \ref{t:amsr} which do not admit a $T^2$ action.

We  now outline the  proof  our main result, Theorem \ref{t:amsr}. 
By Proposition \ref{p:orbitspace}, either (1) the action is fixed-point homogeneous; or (2) the fixed-point set is comprised of two or three isolated points.

Case (1) is covered by Proposition \ref{p:fphcase}. In Case (2), Theorem \ref{t:3points} covers actions with three fixed points, while those with two fixed points are addressed by Propositions \ref{p:4suspension} and \ref{p:loopandspur}.

\subsection{The structure of the orbit space} 
We begin this subsection by demonstrating that
when a circle acts on a $4$-dimensional Alexandrov space
the structure of the orbit space
  is that of a stratified $3$-manifold with the singular strata consisting of the non-principal orbits.

Recall by Lemma \ref{l:fixedpoint}, that when the $4$-space is positively curved the fixed-point set of the circle action is non-empty.  We describe this fixed-point set and investigate the structure of the orbit space when there are only isolated fixed points.
We  show in Lemma \ref{l:s3orbit} that in this case the quotient space is homeomorphic to $\sph^3$ and there are no closed curves corresponding entirely to points of finite isotropy. Lemma \ref{l:atleasttwo} combined with Proposition \ref{p:eulerbound} gives us that the number of fixed points must be two or three.

We fix the following notation. Let $\pi \colon X^4 \rightarrow \bar{X^3}=X^4/T^1$. Let $F$ denote the set of fixed points and let $E$ denote the set of points of finite isotropy in $X^4$. Their images in $\bar{X}^3$ are denoted by $\bar{F}$ and $\bar{E}$ respectively.

\begin{lemma}\label{l:graph}
Let $T^1$ act isometrically and effectively on $X^4$, a $4$-dimensional, closed, orientable Alexandrov space. Then $\bar{X}^3$ is a topological $3$-manifold and the decomposition of $\bar{X}$ by orbit type gives a stratification into manifolds.  Moreover, we can describe the structure of the stratification as follows:
\begin{enumerate}
\item $\bar{F}$ is the union of the (possibly empty) boundary of $\bar{X}^3$ with a (possibly empty) set of isolated points in the interior of $\bar{X}^3$; and
\item $\bar{E}$ is the union of a (possibly empty) set of curves in the interior of $\bar{X}^3$. These curves are locally flat submanifolds of $\bar{X}^3$ and have endpoints in $\bar{F}$ or are simple closed curves. No more than three curves in $\bar{E}$ can intersect at a point of $\bar{F}$.
\end{enumerate}
In other words, the strata are $\bar{X}_0 = \bar{F} \setminus \partial \bar{X}$, $\bar{X}_1 = \bar{E}$, $\bar{X}_2 = \partial \bar{X} \setminus \cl(\bar{E})$ and $\bar{X}_3 = \bar{X} \setminus (\bar{F}\cup\bar{E})$.
\end{lemma}
\begin{proof}

In order to show that $\bar{X}^3$ is a topological manifold, it suffices to show that at every point the space of directions is either $\sph^2$ or $\disk^2$.  

Orientable Alexandrov spaces of dimension three or less are topological manifolds (see
Exercise 10.10.4 Part (2) of~\cite{BBI}).
Thus, the space of directions normal to any orbit in $X^4$ is an orientable, positively curved manifold. Additionally, the isotropy action on it is orientation-preserving and so the quotient is an orientable, positively curved $2$-space, that is, $\sph^2$ or $\disk^2$. But the space of directions at a point in $\bar{X}^3$ is precisely this quotient, proving the claim.

Note further that the only time $\disk^2$ arises as a space of directions is when the isotropy group is $T^1$. This shows that $\partial \bar{X}^3$ is a subset of $\bar{F}$.

	At a fixed point $p\in X^4$, the isotropy group $T^1$ acts on the $3$-dimensional space of directions $\Sigma_p^3$.  
		According to Proposition \ref{p:3space}, the action is the quotient of a linear action on a sphere. The fixed-point set of the isotropy action is therefore empty, or a circle. Therefore each component of $\bar{F}$ has corresponding dimension $0$ or $2$.  The components of dimension $0$ are isolated fixed points, whose spaces of directions are homeomorphic to $\sph^2$ and thus are interior points. The components of dimension $2$ make up the boundary of $\bar{X}^3$, thus proving Part ($1$).

To prove Part ($2$), consider a point $p$ such that $T^1_p = \zzz_k$. The orbit $T^1(p)$ is a circle, and so the normal space is homeomorphic to $\sph^2$.
The effective, isotropy action of $\zzz_k$ is orientation-preserving, so the action is by a rotation, fixing two points.
Hence,  $\Sigma_{\bar{p}}\bar{X} \cong \sph^2$ and it has two points which make up $\Sigma_{\bar{p}}\bar{E}$, so that $\bar{E}$ is a locally flat $1$-dimensional manifold without boundary, with isotropy constant on connected components.

Now we consider the endpoints of the connected components of $\bar{E}$ which are open arcs. This is the set $\cl(\bar{E}) \setminus \bar{E}$. Since points in $\cl(\bar{E}) \setminus \bar{E}$ are, by continuity, fixed by a cyclic subgroup but are not themselves in $\bar{E}$, they must be fixed by the entire circle, so $\partial \bar{E} \subset \bar{F}$.

 Since a circle action on a positively curved $3$-space can have no more than three components of finite isotropy, the bound on the number of intersecting curves is obtained.
\end{proof}

\begin{remark}In the case where $X^4$ is a manifold, $\Sigma^3_p \cong \sph^3$ for every point $p$. If $p$ is a fixed point, then the isotropy action of $T^1$ on $\Sigma^3_p$ cannot have both fixed points and points of finite isotropy, which implies that $\bar{F} \cap \bar{E}$ is comprised only of isolated points of $\bar{F}$.
	
However, in general $\Sigma^3_p$ could have the type of any spherical $3$-manifold. In case $\Sigma^3_p$ is a lens space, the isotropy action can have both fixed points and a circle orbit of finite isotropy. In this case, arcs of $\bar{E}$ can terminate in the boundary of $\bar{X}^3$.
\end{remark}

The case where $\partial \bar{X} \neq \emptyset$ corresponds to fixed-point-homogeneous actions, and these are classified in Section \ref{s:fph}. For the remainder of this subsection, we restrict ourselves to the case where $\partial \bar{X} = \emptyset$, so that the fixed-point set is discrete. In this case, the strata are simply $\bar{X}_0 = \bar{F}$, $\bar{X}_1 = \bar{E}$ and $\bar{X}_3 = \bar{X} \setminus (\bar{F}\cup\bar{E})$.

By ruling out the possibility that $\bar{E}$ could contain simple closed curves, the next Lemma shows that, assuming positive curvature, the singular strata form an embedded multigraph in $\sph^3$ having maximal degree at most three. Recall that a {\em multigraph} is  a graph where two vertices may be joined by multiple edges, or where an edge may join a vertex to itself to form a  loop. 

\begin{lemma}\label{l:s3orbit}
Let $T^1$ act isometrically and effectively on $X^4$, a $4$-dimensional, positively curved, closed, orientable Alexandrov space, with only isolated fixed points. 
Then $\bar{X}^3$ is homeomorphic to $\sph^3$ and $\bar{E}$ contains no simple closed curves.
\end{lemma}

\begin{proof}
By Lemma \ref{l:graph}, $\bar{X}^3$ is a closed manifold. By the Generalized Synge Theorem \ref{GST}, $X^4$ is simply connected and it follows by Corollary II.6.3 of~\cite{Br}, that $\bar{X}^3$ is simply connected. 
By the resolution of the Poincar\'{e} Conjecture it is homeomorphic to $\sph^3$.  

Lemma 2.3 of Montgomery and Yang~\cite{MY1} then guarantees that there are no simple closed curves in $\bar{E}$. Since this Lemma is not given in precisely the correct context, and the notation there conflicts confusingly with ours, we repeat the argument here. Suppose $\bar{c} \subset \bar{E}$ is a closed curve, corresponding to isotropy $\zzz_k$. Let $\bar{Q} \subset \bar{X}$ be an oriented surface bounded by $\bar{c}$, and assume, without loss of generality, that $\bar{Q} \cap \bar{F} = \emptyset$ and $\left( \bar{Q} \,\setminus\, \bar{c} \right) \cap \bar{E}$ is finite. 

Then $Q \,\setminus \,c$, after a closed submanifold of dimension $1$ corresponding to finitely many exceptional orbits is removed, is an orientable $3$-manifold. 
Let $z$ generate $H_3 \left( Q, Q \cap E ; \zzz \right) \cong \zzz$. Then $\partial z \in H_2 (c; \zzz)$. We can see that $\frac{1}{k} \partial z$ is also an integral cycle on $c$, but it does not bound on $Q$, since $z$ is not divisible by $k$.  This shows that $Q$ has $2$-torsion, which is a contradiction.
\end{proof}

The following lemma provides a lower bound of two on the number of isolated fixed points.	In the Riemannian case, the simplest method of proof for Lemma \ref{l:atleasttwo} is to use the fact that the Euler characteristic of a Riemannian manifold or orbifold, which in the simply connected case in dimension $4$ is at least two, is given by the Euler characteristic of the fixed-point set of any isometric circle action. The proof of this (see Kobayashi~\cite{Ko}) relies on the Lefschetz Fixed-Point Theorem, and so ultimately on the triangulability of Riemannian spaces. 
	
		The question of whether a general Alexandrov space is triangulable remains open. Since $\sph^3$ is triangulable, it is probably possible to lift a triangulation of $\sph^3$ to $X^4$, but we do not investigate this question here.

\begin{lemma}\label{l:atleasttwo}
	Let $T^1$ act isometrically and effectively on $X^4$, a $4$-dimensional, positively curved, closed, orientable Alexandrov space, with only isolated fixed points. 
	Then there are at least two fixed points.
\end{lemma}

\begin{proof}
	By Lemma \ref{l:s3orbit}, the orbit space is $\bar{X}^3 \cong \sph^3$ and by Lemma \ref{l:fixedpoint}, there is at least one fixed point. We  assume that there is exactly one fixed point, $p$, to derive a contradiction. There are two cases to consider: Case (1), where there is finite isotropy, and Case (2), where there is none. In both cases, we  show that the space of directions at $p$, $\Sigma_p$, is homeomorphic to  $\sph^1 \times \sph^2$. Since $\sph^1 \times \sph^2$ has infinite fundamental group, by the Bonnet--Myers theorem it does not admit positive curvature, yielding the contradiction.
	
	We consider Case (1), where  the action has points of finite isotropy. Then by Lemmas \ref{l:graph} and \ref{l:s3orbit} the singular strata are given by a multigraph. It has only one vertex, $\bar{p}$, and the vertex has degree at most three. It follows that the multigraph can contain only one loop. This unique closed curve is denoted by $\bar{c}$.

	Then, following the notation of Section \ref{s:invariants}, let $\bar{U} \cong \sph^1 \times \disk^2$ be the closure of a neighborhood of $\bar{c}$ and let $\bar{W} = \cl \left( \bar{X} \setminus \bar{U} \right)$.
	Using the Mayer--Vietoris sequence applied to $\bar{X}=\bar{U}\cup \bar{W}$, we see that $\bar{W}$ has the integral homology of a circle. Since $\bar{c} = \bar{E} \cup \bar{F}$, the action on $W$ is free and so a section $\partial \bar{U} \to \partial U$ can be specified.
	The Seifert invariants of the isotropy action on $\Sigma_p$ with respect to this section are then $\left\lbrace 0 ; (\alpha_i, \beta_i) , (\alpha_i, -\beta_i) \right\rbrace$, so that $\Sigma_p \cong \sph^1 \times \sph^2$ and we obtain the desired contradiction.
	
	We now consider Case (2), where $T^1$ acts freely on the complement of $p$. Let $\epsilon > 0$ be so small that $B_{\epsilon}(p)$ is homeomorphic to the cone on $\Sigma_p X$ and $B_{\epsilon}(\bar{p}) \cong \disk^3$. By Theorem 4.4 of~\cite{HS} and for sufficiently small $\epsilon$,   $\partial (B_{\epsilon}(\bar{p}))\cong\sph^2$ and admits a collared neighborhood. Now, since  $\bar{X}^3\cong\sph^3$ by Lemma \ref{l:s3orbit}, it follows from the Generalized Schoenflies Theorem (see, for example, Brown~\cite{Bro}), that  $\bar{X}^3\setminus B_{\epsilon}(\bar{p})\cong\disk^3$.
	
	The circle $T^1$ acts freely on the complement of $B_{\epsilon}(p)$, $U=X^4\setminus B_{\epsilon}(p)$.  
	Hence $U$ is the total space of a principal $T^1$-bundle over $\disk^3$ and is homeomorphic to $\sph^1 \times \disk^3$. Then $\Sigma_p X^4 \cong \sph^1 \times \sph^2$, and, once again, we obtain the desired contradiction.
	
	Hence there are at least two fixed points.
\end{proof}

 We  summarize the situation with the following proposition, which follows from the foregoing Lemmas, except for the upper bound, which follows from Proposition \ref{p:fixbound}.
 
\begin{proposition}[Orbit space structure]\label{p:orbitspace}
Let $T^1$ act isometrically and effectively on $X^4$, a $4$-dimensional, positively curved, closed, orientable Alexandrov space. Then either
\begin{enumerate}
\item The action is fixed-point homogeneous; or
\item The orbit space $\bar{X}^3$ is homeomorphic to $\sph^3$ and the non-principal orbits are represented by an embedded multigraph of maximal degree three, having at least two vertices. The vertices correspond to $\bar{F}$ while the edges correspond to $\bar{E}$.
\end{enumerate}
In the event that the action satisfies Condition \q{} and the fixed-point set is discrete, there are at most three fixed points.
\end{proposition}

\subsection{The fixed-point-homogeneous case}\label{s:fph}

In this subsection we classify the fixed-point-homogeneous circle actions on positively curved Alexandrov 4-spaces.

\begin{proposition}\label{p:fphcase} Let $T^1$ act isometrically and effectively on $X^4$,  a positively curved, compact, $4$-dimensional Alexandrov space, in a fixed-point-homogeneous manner. Then,  
 $X^4$ is equivariantly homeomorphic to either the spherical suspension of $\sph^3/\zzz_k$ or to a finite quotient of a weighted complex projective space with a  linear $T^1$ action.
\end{proposition}

\begin{proof}
Let $T^1(p)$ be the orbit furthest from $F^2$, the unique codimension two component of $\Fix (X, T^1)$,  Then by Theorem \ref{t:fph}, we have 
\bdm
X^4 \cong (\nu * T^1)/T^1_p,
\edm
where $T^1_p$ acts on the left on $\nu*T^1$, the action on $\nu$ being the isotropy action at $p$ and the action on $T^1$ being the inverse action on the right.
The $T^1$-action on $(\nu * T^1)/T^1_p$ is induced by the left action of $T^1$ on itself.

In the case where $T^1_p$ is finite, $\nu \cong \sph^2$, and so $X^4 \cong (\sph^2 * \sph^1) / \zzz_k \cong \Susp (\sph^3 / \zzz_k)$.

In the case where $T^1_p = T^1$, $X^4 \cong (\nu^3 * T^1)/T^1$ for some positively curved $\nu$. 
By Proposition \ref{p:3space}, $\nu^3 \cong \sph^3 / \Gamma$ for some $\Gamma \subgp \SO(4)$ and the circle action is induced by a linear one.
We claim that the space must be an orbifold. Since $\nu^3 * T^1 \cong \sph^5/\Gamma$ is certainly an orbifold, it is enough to check that $T^1_p$ does not fix points of $\nu^3$. If it did, the fixed-point set would be of codimension two in $\nu^3$, so that $p\in F^2$, contradicting the original choice of $p$. This proves the claim. These orbifolds are finite quotients of $\wcp$.
\end{proof}

\subsection{Three isolated fixed points}\label{s:3points}

In the case where there are three isolated fixed points, we determine the structure of $\bar{E}$ and $\bar{E}\cup \bar{F}$  
and then use Section \ref{s:top3points} to obtain the classification.

\begin{lemma}\label{l:3points} Let $T^1$ act isometrically and effectively  on $X^4$,  a positively curved, closed, orientable, $4$-dimensional Alexandrov space, so as to satisfy Condition \q{}. If there are three isolated fixed points, then $\bar{E} \subset \bar{X}^3$ comprises at most one curve between each pair of points, and those curves are locally flat. If $\bar{E} \cup \bar{F}$ is a closed curve, then it contains all three fixed points, and is unknotted.\end{lemma}

\begin{proof} The proof makes use of the results and methods of Section \ref{s:knots}. 
	
	Local flatness follows from $\bar{E}$ being an extremal set (see Section \ref{s:equivalex}). If a pair of points were joined by more than one curve, any two such curves would generate a closed curve $c$. The third point would then appear twice in the double branched cover over $c$, $X^3_2(c)$, and so $X^3_2(c)$ would have at least four singular points with small spaces of directions, contradicting Proposition \ref{p:eulerbound}.
	Therefore any closed curve must contain all three fixed points. Denoting the curve again by $c$, note that $(\bar{X}^3, c)$ satisfies Condition O (see Definition \ref{d:condition_o}) and an application of Theorem \ref{knot} shows that the curve is unknotted. 
\end{proof}

In case the finite isotropy does not form a closed curve, we may choose additional arcs between the fixed points so that the entire singular set $\bar{E} \cup \bar{F}$ still lies on an unknotted closed curve. Furthermore, we may assume that the curve is locally flat. By adding these virtual edges, we take the multigraph representing the singular set to be the complete graph $K_3$, and treat all spaces with three fixed points in a unified way.

The problem of identifying the Alexandrov space now lies in the context addressed by Section \ref{s:top3points}. Applying Theorem \ref{t:wcp}, we have the following result.

\begin{theorem}\label{t:3points}
Let $T^1$ act isometrically and effectively  on $X^4$,  a positively curved, closed, orientable Alexandrov space, so as to satisfy Condition \q{}. If there are three isolated fixed points, then $X^4$ is a finite quotient of a weighted complex projective space. Furthermore, the circle action can be extended to an action by homeomorphisms of $T^2$, which is induced from the standard $T^3$ action on $\sph^5$.
\end{theorem}

\begin{remark} In contrast to Theorem \ref{t:3points}, when there are just two isolated fixed points, not all $T^1$ actions  extend, including  those actions  that  are fixed-point-homogeneous. 
\end{remark}

\subsection{Two isolated fixed points}\label{s:twofixedpoints}

We begin this section by analyzing the possible configurations of the singular strata 
in the orbit space. We find two broad cases, which are investigated separately. 

\begin{lemma}\label{l:2pointstructure}
	Let $T^1$ act isometrically and effectively  on $X^4$,  a positively curved, closed, orientable, $4$-dimensional Alexandrov space so as to satisfy Condition \q{}, with fixed-point set consisting of two isolated points. Then the orbit space $\bar{X}^3$ is homeomorphic to $\sph^3$  and the singular strata $\bar{E} \cup \bar{F}$ are given by:
	\begin{enumerate}
		\item Two discrete points, $\bar{F}$, and at most three curves, $\bar{E}$, each of which joins the two points, so that the graph is unknotted, that is, the orbit space is homeomorphic as a stratified set to the suspension of a $2$-sphere with at most three points in its $0$-stratum; or
		\item Two discrete points, $\bar{F}$, and either one or two curves, $\bar{E}$, of which one is an unknotted loop based at one of the points while the other, should it exist, joins the two points.
	\end{enumerate} 
\end{lemma}

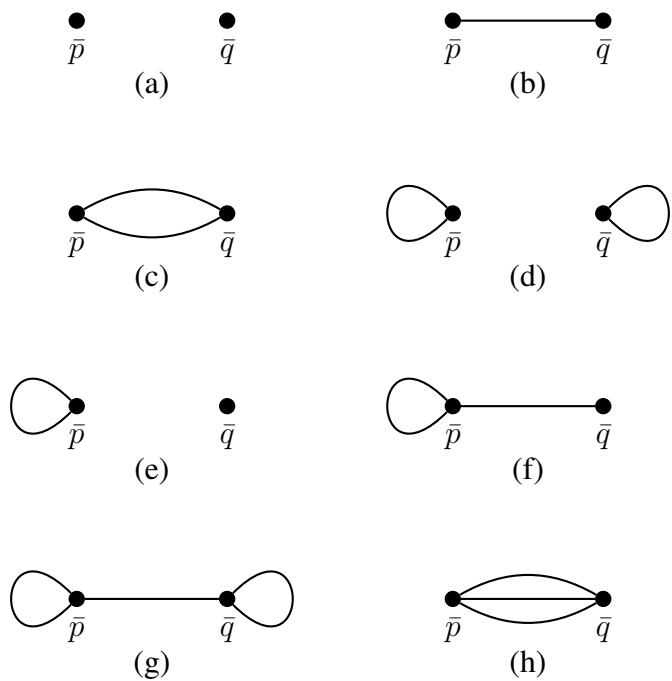
\begin{figure}[t]
	\begin{tabular}{c@{\hspace{2cm}}c}
		\begin{tikzpicture}
		\useasboundingbox (-0.5,-0.475) rectangle (2.5,.95);
		\Vertex[L=\bar{p}]{p}
		\Vertex[L=\bar{q},x=2,y=0]{q}
		\end{tikzpicture}
		& 
		\begin{tikzpicture}
		\useasboundingbox (-0.5,-0.475) rectangle (2.5,.95);
		\Vertex[L=\bar{p}]{p}
		\Vertex[L=\bar{q},x=2,y=0]{q}
		\Edge(p)(q)
		\end{tikzpicture}
		\\
		(a) & (b) \\[0.5cm]
		\begin{tikzpicture}
		\useasboundingbox (-0.5,-0.475) rectangle (2.5,.95);
		\Vertex[L=\bar{p}]{p}
		\Vertex[L=\bar{q},x=2,y=0]{q}
		\Edge[style={bend left}](p)(q)
		\Edge[style={bend right}](p)(q)
		\end{tikzpicture}
		& 
		\begin{tikzpicture}
		\useasboundingbox (-0.5,-0.475) rectangle (2.5,.95);
		\Vertex[L=\bar{p}]{p}
		\Vertex[L=\bar{q},x=2,y=0]{q}
		\Loop[dist=1.5cm,dir=EA,style={thick}](q)
		\Loop[dist=1.5cm,dir=WE,style={thick}](p)
		\end{tikzpicture}
		\\
		(c) & (d) \\[0.5cm]
		\begin{tikzpicture}
		\useasboundingbox (-0.5,-0.475) rectangle (2.5,.95);
		\Vertex[L=\bar{p}]{p}
		\Vertex[L=\bar{q},x=2,y=0]{q}
		\Loop[dist=1.5cm,dir=WE,style={thick}](p)
		\end{tikzpicture}
		& 
		\begin{tikzpicture}
		\useasboundingbox (-0.5,-0.475) rectangle (2.5,.95);
		\Vertex[L=\bar{p}]{p}
		\Vertex[L=\bar{q},x=2,y=0]{q}
		\Edge(p)(q)
		\Loop[dist=1.5cm,dir=WE,style={thick}](p)
		\end{tikzpicture}
		\\
		(e) & (f) \\[0.5cm]
		\begin{tikzpicture}
		\useasboundingbox (-0.5,-0.475) rectangle (2.5,.95);
		\Vertex[L=\bar{p}]{p}
		\Vertex[L=\bar{q},x=2,y=0]{q}
		\Edge(p)(q)
		\Loop[dist=1.5cm,dir=EA,style={thick}](q)
		\Loop[dist=1.5cm,dir=WE,style={thick}](p)
		\end{tikzpicture}
		& 
		\begin{tikzpicture}
		\useasboundingbox (-0.5,-0.475) rectangle (2.5,.95);
		\Vertex[L=\bar{p}]{p}
		\Vertex[L=\bar{q},x=2,y=0]{q}
		\Edge[style={bend left}](p)(q)
		\Edge[style={bend right}](p)(q)
		\Edge(p)(q)
		\end{tikzpicture}
		\\
		(g) & (h) \\
	\end{tabular}
	\caption{Multigraphs on two vertices with maximal degree at most three. 	}
	\label{f:multigraphs}
\end{figure}

\begin{proof}
	From Lemma \ref{l:graph} and Proposition \ref{p:orbitspace}, we know that $\bar{X}^3$ is homeomorphic to a $3$-sphere. The non-principal orbits project to a multigraph in $\bar{X}^3$ which has two vertices, corresponding to fixed points, and has maximal degree at most three.
	
	Since the circle action satisfies Condition \q{}, the spaces of directions at the vertices are small, and for any closed curve $c$ in the multigraph, $(\bar{X}^3, c)$ satisfies Condition O. It follows that any closed curve in the multigraph is unknotted.
	
	We consider the possible configurations by maximal degree of the multigraph, all of which are shown in Figure \ref{f:multigraphs}. 
	When the maximal degree is zero, there are two disconnected vertices (a). 
	If it is one, the vertices are joined by an edge (b). Both of these are covered under Item (1).

	When the maximal degree is two, there are three possible multigraphs. There may be two edges, each connecting the two vertices (c). This space is covered under Item (1). There may be two loops, each based at a separate vertex (d). By taking a double branched cover over one loop, and then over the other, we obtain a space with at least four points with small spaces of directions, in violation of Proposition \ref{p:eulerbound}. 
	
	The final possibility in maximal degree two is that there may be only one loop, $c$, based at one vertex, $\bar{p}$, while the other vertex, $\bar{q}$, is disconnected (e). This case falls under Item (2). By adding a virtual edge between the vertices, we can treat this case 
	in a unified manner with (f).

	When the maximal degree is three, there are again three possible graphs. If there are two edges, then one edge forms a loop while the other joins the two vertices (f), as described under Item (2).  If there are three edges, then two edges form loops, one at each vertex (g). This is ruled out by the argument in the previous paragraph for the case of two loops. Finally, the three edges may each join the two vertices (h). In this case, the Equivariant Suspension Theorem \ref{t:est} shows that the resulting graph is unknotted, so that the case falls under Item (1).
\end{proof}

We first consider the orbit spaces which fall under Item (1) in Lemma \ref{l:2pointstructure}, which are those shown in Figure \ref{f:multigraphs} (a--c, h). We obtain the following proposition by applying 
Propositions \ref{p:suspension} and \ref{p:3space}.

\begin{proposition}[Suspension case]\label{p:4suspension}
	Let $T^1$ act isometrically and effectively  on $X^4$,  a positively curved, closed, orientable, $4$-dimensional Alexandrov space with fixed-point set consisting of two isolated points. Suppose that $\bar{X}^3$ is homeomorphic as a stratified set to a suspension, that is, it is homeomorphic to $\sph^3$, $\bar{F}$ is two discrete points and $\bar{E}$ is at most three curves, each of which joins the two points of $\bar{F}$, so that the singular strata are unknotted.
	
	Then $X^4$ is equivariantly homeomorphic  to the suspension of  $\sph^3 / \Gamma$ with a linear $T^1$ action for some finite $\Gamma \subgp \SO(4)$.
\end{proposition}

Finally, we turn to the orbit spaces given by Item (2) of Lemma \ref{l:2pointstructure}, which are those shown in Figure \ref{f:multigraphs} (e--f). 

As part of this investigation we  need to make use of the fundamental groupoid.
Recall that this is an enhanced version of the fundamental group allowing for multiple base points.
If $A$ is a set of base points in a topological space $X$ we write the fundamental groupoid as $\pi_1(X,A)$.
The groupoid is made up of the homotopy classes of paths between the base points, including the loops.

The fundamental groupoid is often represented in terms of category theory: the base points provide the objects of the category and the paths between the base points are the morphisms. Composition of morphisms is the concatenation of paths. All morphisms are invertible, since paths can be traversed backwards.
This permits a neat statement of the Seifert--van-Kampen Theorem which does not require the intersections of the two covering subsets to be path connected.
The following version is stated and proved in Section 1.6 in \cite{BHS}\footnote{The relevant extract is available at {\scriptsize\url{https://groupoids.org.uk/pdffiles/vKT-proof.pdf}}.}. 

\begin{theorem}\label{t:groupoidSVK}\cite{BHS} Let $X_1$, $X_2$,  be open subsets of a topological space $X$, with $X=X_1\cup X_2$. Let $X_0=X_1\cap X_2$ and let  $A$ be a subset of
$X$ meeting each path component of $X_1, X_2$, and $X_0$ (and therefore of $X$). Let $A_i = X_i \cap A$
for $i = 0, 1, 2$. Then the following diagram of morphisms induced by inclusion
$$
\xymatrix{
 \pi_1(X_0, A_0)  \ar[r]^{a_1}  \ar[d]^{a_2} & \pi_1(X_1, A_1)   \ar[d]^{b_1} \\
 \pi_1(X_2, A_2) \ar[r]^{b_2} & \pi_1(X,A)
}
$$
is a pushout of groupoids.
\end{theorem}
We now consider the orbit spaces which fall under Item (2) in Lemma \ref{l:2pointstructure}, which are those shown in Figure \ref{f:multigraphs} (e, f).
\begin{proposition}[Loop-and-spur case]\label{p:loopandspur}  Let $T^1$ act isometrically and effectively  on $X^4$,  a positively curved, closed, orientable, $4$-dimensional Alexandrov space with fixed-point set consisting of two isolated points. Suppose that the orbits of exceptional isotropy project to one or two curves, $\bar{E} \subset \bar{X}$, of which one is an unknotted loop based at one of the points while the other, should it exist, joins the two points.
		
	Then $X^4$ is equivariantly homeomorphic to a finite quotient of a weighted complex projective space with a $T^1$ action induced by a linear action on $\sph^5$. 
\end{proposition}

\begin{proof}
	To unify the treatment, if $\bar{E}$ contains only one curve, we add the geodesic $\bar{p} \bar{q}$ to the graph as a ``virtual edge'' which corresponds to orbits with isotropy $\zzz_1$. The multigraph given by the singular set is then that with two vertices, one loop, and one edge. We denote the loop by $c$, the vertex on $c$ by $\bar{p}$, and the other vertex by $\bar{q}$. 
	
We  show that $X^4$ is the quotient by an involution of some space such that the lifted circle action has three fixed points. Theorem \ref{t:3points} then gives us the result. 

In order to do this we  endow $X^4$ with an orbifold structure. The only topological singularities in $X^4$ are at the isolated fixed points of the action. Neighborhoods of these points can be given charts with a local group $\Gamma \subset SO(4)$ such that $\Gamma$ is isomorphic to the fundamental group of the space of directions. All other points are manifold points 
and we take the local group there to be trivial.
	
Now we can calculate the orbifold fundamental group of $X^4$, $\pi_1^{\mathrm{orb}} (X^4)$.
Using a transversality argument, one sees that $\pi_1^{\mathrm{orb}}(X)$ is determined by the regular part of $X$ and singularities of codimension one and two only (see Theorem A.I.4 in \cite{HQ}). So, since $X^4$ has no singularities of codimension two, $\pi_1^{\mathrm{orb}} (X^4) \cong \pi_1(X^4_{\mathrm{reg}})$, where $X^4_{\mathrm{reg}}$ is the regular part of the orbifold.

We decompose $\bar{X}^3$ into four regions as shown in Figure \ref{f:orbifoldgroup}, which lifts to a decomposition of $X^4$.
It simplifies this discussion to treat the sets as though they were closed sets which meet along their boundary.
In order to apply the Seifert--van-Kampen Theorem, we need an open cover of $X^4$. However, using closed sets does not generate any problems, 
provided the boundaries of these sets have collar neighborhoods.
	
		\begin{figure}
		\begin{tikzpicture}
		\node[anchor=south west,inner sep=0] (image) at (0,0) {\includegraphics[width=0.77\textwidth]{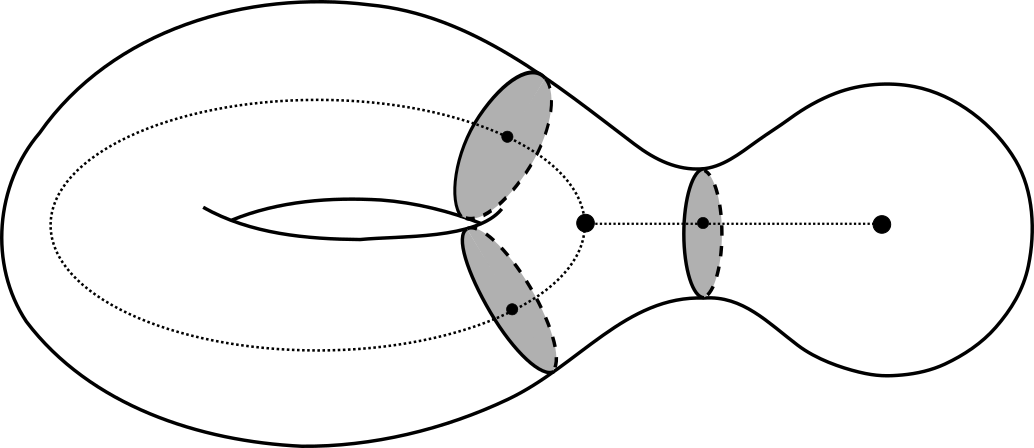}};
		\begin{scope}[
		x={(image.south east)},
		y={(image.north west)}
		]
		\node  at (0.6,0.4) {$\bar{A}$};
		\node  at (0.88,0.3) {$\bar{B}$};
		\node  at (0.25,0.4) {$\bar{C}$};
		\node  at (0.68,0.16) {$\bar{D}$};
		\node  at (0.54,0.5) {$\bar{p}$};
		\node  at (0.88,0.5) {$\bar{q}$};
		\node  at (0.03,0.5) {$c$};
		\end{scope}
		\end{tikzpicture}
		\caption{Decomposition of $\bar{X}^3 \cong \sph^3$ into four closed sets, where $\bar{D}$ is the complement of
		the solid torus, $\bar{A}\cup\bar{B}\cup \bar{C}$, viewed in $S^3\setminus\{\infty\}$.}
		\label{f:orbifoldgroup}
	\end{figure}

	Let $\bar{A}$ be a conical neighborhood of $\bar{p}$ and let $\bar{B}$ be a conical neighborhood of $\bar{q}$ so that they intersect in a $\disk^2$ which is transverse to the edge $\bar{p} \bar{q}$. Let $\bar{C}$ be a neighborhood of $\left( c \,\setminus\, \bar{A} \right)$ homeomorphic to $[0,1] \times \disk^2$, intersecting $\bar{A}$ at $\left\lbrace 0 , 1 \right\rbrace \times \disk^2$ but disjoint from $\bar{B}$. Let $\bar{D}$ be $\bar{X}\, \setminus \,\left( \bar{A} \cup \bar{B} \cup \bar{C} \right)$. Note that $\bar{D} \cong \sph^1 \times \disk^2$. 
	Let $A$, $B$, $C$ and $D$ be the preimages of each of $\bar{A}$, $\bar{B}$, $\bar{C}$, and $\bar{D}$, respectively.

{\bfseries Calculating $\pi_1(A_{\mathrm{reg}} \cup B_{\mathrm{reg}})$.} The sets $A$ and $B$ are both cones on spherical manifolds.
It follows that $A_{\mathrm{reg}}$ and $B_{\mathrm{reg}}$ are homotopy equivalent to those spherical manifolds.
We  calculate the fundamental groups of those 3-manifolds from their  Seifert invariants.

Note that $\pi:D\rightarrow \bar{D}$ is a principal $S^1$ bundle, and recall that these are classified by $H^2(\bar{D})=0$. 
Therefore, there exists a section $\sigma \colon \bar{D} \to D$. 
This section can be used to define the invariants, as described in Sections \ref{s:seifert} and \ref{s:invariants}. Let $q_1$ and $q_2$ be the two components of $\sigma \left( \partial \left(\bar{A} \cap \bar{C} \right)\right)$, where $\bar{A}\cap \bar{C}$ is oriented as a submanifold of $\partial \bar{A}$. Let $q_3$ be $\sigma \left( \partial \left( \bar{A} \cap \bar{B} \right) \right)$, again with $\bar{A}\cap \bar{B}$ oriented as a submanifold of $\partial \bar{A}$.
	
	Without loss of generality, $\sigma$ may be chosen so that the invariants of $\partial A$ are $\left\{0;  (k, -1), (k,1), (\alpha, \beta)  \right\}$ and the invariants of $\partial B$ are then $\left\{0;  (\alpha, -\beta)  \right\}$. Where the edge in the graph is only virtual, $\alpha=1$. Since $\partial A$ cannot be $\sph^2 \times \sph^1$, we have $\beta \neq 0$. By assumption, the loop is not virtual and so $k \geq 2$.
	
	 We may also assume that the three pairs of invariants correspond to the loops $q_1, q_2, q_3$ in that order. Since the first pair is $(k,-1)$, the curve $q_1$ is, in fact, homotopic to the corresponding exceptional orbit of isotropy $\zzz_k$.
	
	From these invariants, using Theorem \ref{t:pi1}, we obtain the following presentations of the fundamental groups. Letting $h$ be a principal orbit, the group $\pi_1 (A_{\mathrm{reg}}, x)$ is generated by $h, q_1, q_2$ and $q_3$ and is given by 
	$$ \left\langle q_1, q_2, q_3, h \st \left[ h, q_i \right] = 1, q_1^k h^{-1} = 1, q_2 ^k h^{1} = 1, q_3 ^{\alpha} h^{\beta} = 1, q_1 q_2 q_3 = 1 \right \rangle.$$
	
	Note that the section $q_3$ must have its orientation reversed for the calculation of $\pi_1 (B_{\mathrm{reg}}, x)$, so that we get
		$$ \pi_1 (B_{\mathrm{reg}}, x) = \left\langle q_3^{-1}, h \st \left[ h, q_3^{-1} \right] = 1, q_3 ^{-\alpha} h^{-\beta} = 1, q_3^{-1} = 1 \right \rangle = \zzz_{\left| \beta \right|}.$$

When applying the Seifert--van-Kampen Theorem to calculate $ \pi_1 (A_{\mathrm{reg}} \cup B_{\mathrm{reg}}, x)$ there is no need to include any additional relations, since these are encoded in the use of $q_3$ and $h$ as generators of both groups. Taking advantage of $q_3^{-1} = 1$ we immediately obtain $q_2 = q_1^{-1}$. Since $h=q_1^k$ we may discard the relation $\left[ h, q_1 \right] = 1$ as being already implied, but it is convenient later to maintain both generators. The group reduces to
	$$ \pi_1 (A_{\mathrm{reg}} \cup B_{\mathrm{reg}}, x) = 
	\left\langle 
	q_1, h \st \left[ h, q_1 \right] = 1, q_1^k h^{-1} = 1, h^{\beta} = 1 \right \rangle,$$
	which is simply the cyclic group of order $k \left| \beta \right|$, generated by $q_1$. We note once more that $q_1$ is an orbit with isotropy $\zzz_k$ corresponding to the `loop' in $\bar{X}^3$  and $h = q_1 ^k$ is a principal orbit.

{\bfseries Calculating $\pi_1(A_{\mathrm{reg}} \cup B_{\mathrm{reg}} \cup C_{\mathrm{reg}})$.} In order to include $C$, observe that $(A \cup B) \cap C$ has two components.
This means we must make use of the fundamental groupoid rather than the fundamental group.
Let $l$ be a longitudinal curve on the torus $\partial \left( \bar{A} \cup \bar{B} \cup \bar{C} \right)$ such that $l \cap C$, and hence $l \cap (A \cup B)$, is connected. 
We choose $y,z \in \sigma(l)$, one in each component of $(A \cup B) \cap C$, as the basepoints. 

The groupoid $\pi_1 (A_{\mathrm{reg}} \cup B_{\mathrm{reg}}, \lbrace y, z \rbrace)$ has the points $y$ and $z$ as objects. The morphisms $y \to y$ and $z \to z$ are both given by the fundamental group, so they are the cyclic group generated by $q_1$.
The morphisms $y \to z$ are given by the portion of $\sigma(l)$ lying in $A \cup B$, which we call $m$, composed with an element of $\pi_1 (A_{\mathrm{reg}} \cup B_{\mathrm{reg}}, z)$ 
It is clear that $m$ followed by $q_1 \in \pi_1 (A_{\mathrm{reg}} \cup B_{\mathrm{reg}}, z)$ is homotopic to $q_1 \in \pi_1 (A_{\mathrm{reg}} \cup B_{\mathrm{reg}}, y)$ followed by $m$.

Note that $\bar{C}$ is contractible and that $C = C_{\mathrm{reg}}$, so $\pi_1 (C_{\mathrm{reg}}, y)$ is the cyclic group generated by $q_1$. 
The full groupoid is then given by a second copy of the cyclic group for $\pi_1 (C_{\mathrm{reg}}, z)$
with the morphisms $y \to z$ given by the remaining part of $\sigma(l)$, which we call $n$, composed with an element of $\pi_1 (C_{\mathrm{reg}}, z)$. Once more, $n$ commutes with $q_1$.

We now obtain $\pi_1 (A_{\mathrm{reg}} \cup B_{\mathrm{reg}} \cup C_{\mathrm{reg}}, \lbrace y, z \rbrace)$ by Theorem \ref{t:groupoidSVK} as the push-out in the category of groupoids. 
Note that the fundamental groupoid of the intersection, $\pi_1 ((A_{\mathrm{reg}} \cup B_{\mathrm{reg}}) \cap C_{\mathrm{reg}}, \lbrace y, z \rbrace)$, is given by two copies of the cyclic group generated by $q_1$ for morphisms $y \to y$ and $z \to z$ while the set of morphisms $y \to z$ is empty. 
The composition of $m$ followed by $n^{-1}$ introduces a new loop based at $y$, which we  call $a$, that commutes with $q_1$. Clearly $a = \sigma(l)$. We obtain 
$$\pi_1 (A_{\mathrm{reg}} \cup B_{\mathrm{reg}} \cup C_{\mathrm{reg}}, y) = 
	\left\langle 
	a, q_1, h \st \left[ a, q_1 \right] = 1, q_1^k h^{-1} = 1, h^{\beta} = 1 \right \rangle$$
and similarly a new loop is introduced based at $z$.

{\bfseries Calculating $\pi_1(A_{\mathrm{reg}} \cup B_{\mathrm{reg}} \cup C_{\mathrm{reg}} \cup D_{\mathrm{reg}})$.}
	Finally, $\bar{D}$ is a solid torus, with fundamental group generated by $m$, the meridianal curve on $\partial \left( \bar{A} \cup \bar{B} \cup \bar{C} \right)$. 
Letting $b = \sigma(m)$, the group $\pi_1 (D_{\mathrm{reg}}, y)$ (note $D = D_{\mathrm{reg}}$) is generated by $b$ and $h$. The fundamental group of the boundary $\pi_1 (\partial D_{\mathrm{reg}}, y)$ is generated by $b$, $a$ and $h$. Since $a$ is killed by the inclusion $\partial D_{\mathrm{reg}} \to D_{\mathrm{reg}}$ and the inclusion $\partial D_{\mathrm{reg}} \to \left( A_{\mathrm{reg}} \cup B_{\mathrm{reg}} \cup C_{\mathrm{reg}} \right)$ maps $b \mapsto q_1$, we are left with $$\pi_1^{\mathrm{orb}} (X^4, y) \cong \pi_1 (X^4_{\mathrm{reg}}, y) \cong \zzz_{k \left| \beta \right|},$$
	a cyclic group generated by the orbit with isotropy $\zzz_k$, where $k \geq 2$ and $\beta = 0$.
	
{\bfseries Consequences.}	Now take the universal cover of $X^4$, $\widetilde{X}$, and consider the lifted circle action. 
The point $p \in X^4$ has only one lift in $\widetilde{X}$, while $q$ has $k$ lifts. There are therefore $k+1$ fixed points of the circle action on $\widetilde{X}$ so, by Proposition \ref{p:fixbound}, $k+1 \leq 3$ and therefore $k=2$. By Theorem \ref{t:3points} and the triviality of $\pi_1^{\mathrm{orb}} (\widetilde{X})$, $\widetilde{X}$ is some $\wcp$. The space $X^4$ is therefore a quotient by an involution of some $\wcp / \zzz_{\left| \beta \right|}$.

It is now possible to check that the local group at $p$, which is $\pi_1(A_{\mathrm{reg}})$, is in fact the dihedral group of order $4 \alpha | \beta |$. The local group at the lift of $p$ is generated by the loop $q_3 \in \pi_1(A_{\mathrm{reg}})$, which is a cyclic normal subgroup of order $2 \alpha$. The local group at the lifts of $q$ will be trivial, and therefore $\widetilde{X} \cong \ccc P^2_{\pm 2\alpha, 1, 1}$.
	\end{proof}

\begin{example}
Here we explicitly give examples of involutions on weighted complex projective spaces which produce the loop-and-spur configuration discussed in Proposition \ref{p:loopandspur}.

Recall the description of the weighted $\cp^2_{a,b,c}$ in Section \ref{s:cp2}.  Consider the $T^2$ action on $\sph^5\subset \ccc^3$ generated by the circles $T^1_{2n,-1,-1}$ and $T^1_{3n,-2,-1}$, as well as the involution $\iota \colon \sph^5 \to \sph^5$ given by $(z_1, z_2, z_3) \mapsto (\bar{z_1}, \bar{z_3}, -\bar{z_2})$. The quotient of $\sph^5$ by the first circle is the space $\cp^2_{2n,-1,-1}$, and the second circle induces a circle action on $\cp^2_{2n,-1,-1}$ with three fixed points and two components of finite isotropy of order $n$, so that the Seifert invariants are $\left( \frac{0}{1}, \frac{-1}{n}, \frac{1}{n}\right)$.

The involution descends to an involution $\bar{\iota}$ on $\cp^2_{2n,-1,-1}$ which has exactly one fixed point, corresponding to the circle $(z_1,0,0) \subset \sph^5$. The circles $(0,z_2,0)$ and $(0,0,z_3)$, the other two fixed points of the circle action on $\cp^2_{2n,-1,-1}$, are interchanged by $\bar{\iota}$. 

The circle still acts on $\cp^2_{2n,-1,-1} / \bar{\iota}$, and in fact the involution $\iota$ descends all the way to $\cp^2_{2n,-1,-1} / T^1 \cong \sph^3$, where it fixes a circle. Since $\bar{\iota}$ fixed only a single point in $\cp^2_{2n,-1,-1}$, it follows that the circle action on $\cp^2_{2n,-1,-1} / \bar{\iota}$ has a codimension $2$ component of $\zzz_2$ isotropy, and that in the orbit space, $\sph^3$, this maps to a loop based at the fixed point of $\bar{\iota}$. As noted earlier, the other two fixed points of the circle action are identified by $\bar{\iota}$, so the orbit space has the structure of a loop of $\zzz_2$ isotropy and a spur of $\zzz_n$ isotropy.
\end{example}

\section{General Alexandrov spaces}\label{s:alex}

This work has been entirely motivated by the following conjecture.

\begin{conjecture} \label{conj} 
	Let $T^1$ act isometrically and effectively on $X^4$, where $X^4$ is a $4$-dimensional, closed, positively curved, orientable Alexandrov space. Then, up to equivariant homeomorphism, $X$ is one of the following spaces:
	\begin{enumerate}
		\item the suspension of a spherical $3$-manifold, with a linear action; or
		\item a finite quotient of a weighted complex projective space with a linear action.
	\end{enumerate}
\end{conjecture}

The only remaining obstacle to proving this conjecture is the requirement that the action satisfy Condition \q{}. In this section we  review how Condition \q{} is used in proving Theorem \ref{t:amsr} in order to clarify what further work is necessary to remove it.

Recall that there are three conditions for the isotropy action at a fixed point to satisfy Condition \qprime{}:

\begin{enumerate}
	\item $\Sigma^3/T^1$ is a small space;
	\item The double branched cover of $\Sigma^3 / T^1$ over any two points corresponding to finite isotropy is small; and
	\item If there are three components of finite isotropy, $\diam \left( \Sigma^3 / T^1 \right) \leq \frac{\pi}{4}$.
\end{enumerate} 	

As mentioned in the Introduction, while Condition \qprime{} may seem technical, the following lemma shows that it is satisfied for any fixed-point-free  isometric circle action on 
any $3$-dimensional spherical orbifold of constant curvature $1$.

\begin{lemma}\label{l:qprime}

Let $T^1$ act isometrically and without fixed points on a $3$-dimensional spherical orbifold $\Sigma^3$ of constant curvature 1. Then the action satisfies Condition \qprime{}.
\end{lemma}
\begin{proof}

We recall from McGowan \cite{McG} that if $G$ is a compact Lie group acting by isometries on the $n$-sphere $S^n$, and $G_0$ denotes the connected component of the
identity element, then $G/G_0$ acts on $S^n/G_0$ by isometries.
So, for any isometric circle action  on  $\Sigma^3$, we have  $\Sigma^3/T^1=(S^3(1)/T^1)/\Gamma$, where $\Gamma$ is a finite group.
Note that $S^3/T^1$ is either a football orbifold or it is $S^2(1/2)$.

In the case that $S^3/T^1$ is a `football' orbifold, there is a natural distance-decreasing map from $S^2(1/2)$ to $\Sigma^3/T^1$ as well as to the double branched cover over the points of finite isotropy, showing that Conditions (1) and (2) are satisfied.

In the case where $S^3/T^1=S^2(1/2)$, 
then $\curv(\Sigma^3 / T^1) \geq 4$ and the same is true of any double branched cover over points of finite isotropy. By comparison to $\sph^2(1/2)$ (see \cite{GM}), Conditions (1) and (2) must hold. It is clear that if there are three components of finite isotropy, then the group $\Gamma$ is not cyclic. So the diameter bound, Condition (3), holds by Greenwald \cite{Gr} (cf. Dunbar, Greenwald, McGowan, and Searle \cite{DGMS}).
\end{proof}

 The first condition guarantees that if $p$ is a fixed point then the image of the fixed point, $\bar{p} \in X^4 / T^1$, has a small space of directions, $\Sigma_{\bar{p}}$. If $X^4$ is positively curved, then Proposition \ref{p:fixbound} yields the crucial upper bound of three on the number of fixed points.
 
 Where there are three fixed points, the second condition is used in Lemma \ref{l:3points} to show that Condition O is satisfied, which guarantees that any closed curve is unknotted and passes through all three points. Here it is crucial that, for $p$ a fixed point, not only is $\Sigma_{\bar{p}}$ small, but its double-branched cover over two points corresponding to finite isotropy is also small.
 
 Where there are two fixed points, Lemma \ref{l:2pointstructure} uses the second condition in the same way to guarantee that any closed curves are unknotted. The third condition is used to show that $\uptheta$-graphs are unknotted, by permitting the application of Theorem \ref{t:est}.

 Since closed, orientable $3$-dimensional Alexandrov spaces are equivariantly homeomorphic to $\sph^3/\Gamma$ with a linear $T^1$ action, the following conjecture would imply that Condition \qprime{} is always satisfied.
  
   \begin{conjecture}\label{extentconjecture}
 	Let $T^1$ act isometrically and without fixed points on $\Sigma^3=\sph^3/\Gamma$, a closed,  orientable $3$-dimensional Alexandrov space with $\curv \geq 1$.
 	Then
 $$\xt_q(\Sigma^3/T^1)=\xt_q((\sph^3)/\Gamma)/T^1)\leq \xt_q((\sph^3(1)/\Gamma)/T^1),$$
 where $T^1$ acts so that $\sph^3(1)/\Gamma$ and
 $\Sigma^3$ are $T^1$-equivariantly homeomorphic.
 \end{conjecture}

  In $4$-dimensional Riemannian manifolds and orbifolds the spaces of directions are isometric to $\sph^3(1)$ or finite quotients of the same, and so Conjecture \ref{extentconjecture} holds  
  trivially.
  We are motivated in making Conjecture \ref{extentconjecture} by the general principle that spaces with $\curv \geq 1$ are in some sense ``smaller'' than spaces with constant curvature~$1$.

In particular, Theorem A of Grove and Markvorsen~\cite{GM} states that $$\xt_q(X^n)\leq \xt_q(\sph^n(1))$$ for any Alexandrov space $X$ with $\curv \geq 1$. Conjecture \ref{extentconjecture} can then be viewed as the correct equivariant version of Theorem A of~\cite{GM}, at least 
in the particular case of spherical $3$-manifolds.



\begin{thebibliography}{99}


\bibitem{Br} Bredon, G. E. (1972). \emph{Introduction to compact transformation groups.} Pure and Applied Mathematics, {\bf 46}. New York, NY: Academic.

\bibitem{Bro} Brown, M. (1960). \emph{A proof of the generalized Schoenflies theorem.} Bull. Amer. Math. Soc., {\bf 66}, 74--76.

\bibitem{BHS} Brown, R., Higgins, P. J., \& Sivera, R (2011). \emph{Nonabelian algebraic topology: filtered spaces,
crossed complexes, cubical homotopy groupoids.} EMS Tracts in Mathematics, {\bf 15}. Zurich: EMS Publishing House.

\bibitem{BBI} Burago, D., Burago, Y., \& Ivanov, S. (2001). \emph{A course in metric geometry.} Graduate Studies in Mathematics, {\bf 3}. Providence, RI: American Mathematical Society.

\bibitem{BGP} Burago, Y., Gromov, M., \& Perelman, G. (1992). \emph{A.D. Aleksandrov's spaces with 
curvature bounded below} (Russian). Uspekhi Mat. Nauk, {\bf 47}, 3--51, 222; translation in Russian Math. Surveys, {\bf 47}, 1--58.

\bibitem{CMB} Calcut, J. S., \& Metcalf-Burton, J. R. (2016). \emph{Double branched covers of theta-curves.} J. Knot Theory Ramifications, {\bf 25}, 1650046, 9pp.

\bibitem{Ch} Choi, S (2012). \emph{Geometric Structures on 2-orbifolds: Exploration of Discrete Symmetry}. MSJ Memoirs {\bf 27}. Tokyo: Mathematical Society of Japan.

\bibitem{CNZZ} Corro, D., N\'u\~{n}ez-Zimbr\'on, J., \& Zarei, M. (2019). \emph{Torus actions on Alexandrov 4-spaces.} Preprint, arXiv:\linebreak[2]{}1902.09402 [math.DG].

\bibitem{Dunbar} Dunbar, W. D. (1994). \emph{Nonfibering spherical 3-orbifolds.} Trans. Amer. Math. Soc., {\bf 341}, 121--142.

\bibitem{DGMS} Dunbar, W. D., Greenwald, S. J., McGowan, J., Searle, C. (2009). \emph{Diameters of 3-sphere quotients.} Differ. Geom. Appl., {\bf 27}, 307--319.

\bibitem{F1} Fintushel, R. (1977). \emph{Circle actions on simply connected 4-manifolds.} Trans. Amer. Math. Soc., {\bf 230}, 147--171.

\bibitem{Fr} Freedman, M. H. (1982). \emph{The topology of four-dimensional manifolds.} J. Differential Geom., {\bf 17}, 357--453.


\bibitem{GG}Galaz-Garc\'ia, F. (2012). \emph{Nonnegatively curved fixed point homogeneous manifolds in low dimensions.} Geom. Dedicata, {\bf 157} 367--396.

\bibitem{GG2} Galaz-Garc\'ia, F. (2012). \emph{Simply connected Alexandrov 4-manifolds with positive or nonnegative curvature and torus actions.} Preprint, arXiv:\linebreak[2]{}1208.3041 [math.DG].

\bibitem{GGG1} Galaz-Garc\'ia, F., \& Guijarro, L. (2013). \emph{Isometry groups of Alexandrov spaces.} Bull. London Math. Soc, {\bf 45}, 567--579.

\bibitem{GGG2} Galaz-Garc\'ia, F., \& Guijarro, L. (2015). \emph{On three-dimensional Alexandrov spaces.} Int. Math. Res. Notices, {\bf 2015}, 5560--5576.

\bibitem{GGK} Galaz-Garc\'ia, F., \& Kerin, M. (2014). {Cohomogeneity-two torus actions on non-negatively curved
	manifolds of low dimension.} Math. Zeitschrift, {\bf 276}, 133--152.

\bibitem{GGS} Galaz-Garc\'ia, F., \& Searle C. (2011). \emph{Cohomogeneity one Alexandrov spaces.} Transform. Groups, {\bf 16}, 91--107.

\bibitem{Gr} Greenwald, S. (1998). {\em Diameters of spherical Alexandrov spaces and curvature one orbifolds.} Ph.D. Thesis, Department of Mathematics, University of Pennsylvania.

\bibitem{GM} Grove, K., \& Markvorsen, S. (1995). \emph{New extremal problems for the Riemannian recognition problem via Alexandrov geometry.} J. Amer. Math. Soc., \textbf{8}, 1--28.

\bibitem{GS1} Grove, K., \& Searle, C. (1994). \emph{Positively curved manifolds with maximal symmetry rank}, J. Pure Appl. Alg. {\bf 91} 137--142.

\bibitem{GS2} Grove, K., \& Searle, C. (1997). \emph{Differentiable topological restrictions by curvature and symmetry.} J Differential Geom., \textbf{47} 530--559.

\bibitem{GW2} Grove, K., \& Wilking, B. (2014). \emph{A knot characterization and 1-connected nonnegatively curved 4-manifolds with circle symmetry.} Geom. Topol., {\bf 18}, 3091--3110.

\bibitem{HQ} Haefliger, A., \& Quach, N.D. (1984). \emph{Une pr\'esentation du groupe fondamental d'une orbifold.} Ast\'erisque, {\bf 116}, 98--107.

\bibitem{Ham} Hamilton, R.\ S. (1982). \emph{Three-manifolds with positive Ricci curvature.} J. Differential Geom., \textbf{17}, 255--306.

\bibitem{HK} Hsiang, W.-Y., \& Kleiner, B. (1989). \emph{On the topology of positively curved $4$-manifolds with symmetry.} J. Differential Geom., \textbf{30}, 615--621.

\bibitem{HS} Harvey, J., \& Searle, C. (2017). \emph{Orientation and symmetries of Alexandrov spaces with applications in positive curvature.} J. Geom. Anal., {\bf 27}, 1636--1666.

\bibitem{JN} Jankins, M., \& Neumann, W. D. (1983). \emph{Lectures on {S}eifert manifolds}, Brandeis Lecture Notes, {\bf 2}. Waltham, MA: Brandeis University.

\bibitem{K} Kapovitch, V. (2007). \emph{Perelman's stability theorem}.  Surv. Differ. Geom., \textbf{11}, 103--136. Somerville, MA: Int. Press.

\bibitem{Ko} Kobayashi, S. (1958). \emph{Fixed points of isometries.} Nagoya Math. J., \textbf{13}, 63--68.

\bibitem{McG} McGowan, J. (1993). {\em The diameter function on the space of space forms.} Compositio Helv. Math., {\bf 87}, 79--98.

\bibitem{MY1} Montgomery, D., \& Yang, C. T. (1960). \emph{Groups on $\sph^n$ with principal orbits of dimension $n-3$.} Illinois J. Math., {\bf 4}, 507--517.

\bibitem{M}  Mostert, P. S. (1957). {\em On a compact Lie group acting on a manifold}, Ann. Math. (2), {\bf 65}, 447-455.

\bibitem{N}  Neumann, W. D. {\em 3-dimensional $G$-manifolds with 2-dimensional orbits}, 1968 Proc. Conf. on Transformation Groups (New Orleans, La., 1967), 220--222. Berlin: Springer.

\bibitem{NZ} N\'u\~{n}ez-Zimbr\'on, J. (2018). {\em Closed three-dimensional Alexandrov spaces with isometric circle actions.} Tohoku Math. J., {\bf 70}, 267--284.

\bibitem {Orlik} Orlik, P. (1972). {\em Seifert manifolds.} Lecture Notes in Mathematics, {\bf 291}. Berlin: Springer-Verlag.


\bibitem{OR} Orlik, P., \& Raymond, F. (1968). \emph{Actions of $SO (2)$ on 3-manifolds.} Proceedings of the Conference on Transformation Groups (New Orleans, La., 1967), 297--318. Berlin: Springer.

\bibitem{Pa} Palais, R. S. (1960). \emph{The classification of {$G$}-spaces.} Mem. Amer. Math. Soc. No. 36.

\bibitem{P} Perelman, G. (1991). \emph{A.D. Alexandrov's spaces with curvatures bounded from below, II.} Preprint.
 
\bibitem{PePet}  Perelman, G., \& Petrunin, A. (1994). {\em Extremal subsets in Aleksandrov spaces and the generalized Lieberman theorem} (Russian). Algebra i Analiz, {\bf 5}, 242--256; translation in St. Petersburg Math. J., {\bf 5}, 215--227.

\bibitem{Pet1}  Petrunin, A. (1998). \emph{Parallel transportation for Alexandrov spaces with curvature bounded below}. Geom. Funct. Anal., {\bf 8}, 123--148.

 
\bibitem{Pl} Plaut, C. (2002). \emph{Metric spaces of curvature $> k$}. Handbook of geometric topology, 819--898. Amsterdam: North-Holland.

 
\bibitem{Ray} Raymond, F. (1968). \emph{Classification of the actions of the circle on
              {$3$}-manifolds}. Trans. Amer. Math. Soc., \textbf{131}, 51--78.

\bibitem{SY} Searle, C., \& Yang, D.-G. (1994). \emph{On the topology of non-negatively curved simply connected
              {$4$}-manifolds with continuous symmetry}.
   Duke Math. J.,
  \textbf{74}, 547--557.

\bibitem{Scott} Scott, P. (1983). \emph{The geometries of 3‐manifolds.} Bull. London Math. Soc., \textbf{15}, 401--487.

\bibitem{Sh} Shiohama, K. (1993). \emph{An introduction to the geometry of Alexandrov spaces.} Lecture Notes Series, {\bf 8}. Seoul: Seoul National University, Research Institute of Mathematics, Global Analysis Research Center.



\bibitem{Y} Yeroshkin, D. (2014). \emph{On geometry and topology of 4-orbifolds}. Preprint, arXiv:\linebreak[2]{}1411.1700 [math.DG].
\end{thebibliography}
\end{document}